\documentclass[aap,preprint]{imsart}

\usepackage{amsmath, amsthm, amstext, amssymb, amsfonts, graphicx, bm, subfigure}

\RequirePackage[OT1]{fontenc}
\RequirePackage[numbers]{natbib}
\RequirePackage[colorlinks,citecolor=blue,urlcolor=blue]{hyperref}

\arxiv{math.ST/1311.1738}

\numberwithin{equation}{section}

\newcommand{\PR}{\mathbb P}
\newcommand{\QR}{\mathbb Q}
\newcommand{\mytilde}{\raise.17ex\hbox{$\scriptstyle\mathtt{\sim}$}}

\newtheorem{theorem}{Theorem}[section]
\newtheorem{corollary}[theorem]{Corollary}
\newtheorem{definition}[theorem]{Definition}
\newtheorem{lemma}[theorem]{Lemma}
\newtheorem{proposition}[theorem]{Proposition}
\newtheorem*{remark}{Remark}

\begin{document}

\begin{frontmatter}

\title{Asymptotic quantization of exponential random graphs}
\runtitle{Asymptotic quantization of ERGMs}

\begin{aug}
\author{\fnms{Mei}
\snm{Yin}\thanksref{a1}\ead[label=a1]{mei.yin@du.edu}},
\author{\fnms{Alessandro}
\snm{Rinaldo}\thanksref{a2}\ead[label=a2]{arinaldo@cmu.edu}},
\and
\author{\fnms{Sukhada}
\snm{Fadnavis}\thanksref{a3}\ead[label=a3]{sukhada23@gmail.com}}

\thankstext{a1}{Mei Yin's research was partially supported by NSF grant DMS-1308333.}
\thankstext{a2}{Alessandro Rinaldo's research was partially
supported by grant FA9550-12-1-0392 from the U.S. Air Force Office
of Scientific Research (AFOSR) and the Defense Advanced Research
Projects Agency (DARPA) and by NSF CAREER grant DMS-1149677.}

\runauthor{Mei Yin, Alessandro Rinaldo, and Sukhada Fadnavis}

\affiliation{University of Denver\thanksmark{a1}, Carnegie Mellon
University\thanksmark{a2}, and Harvard University\thanksmark{a3}}

\address{Department of Mathematics\\
University of Denver\\
\printead{a1}\\
\phantom{E-mail:\ }}

\address{Department of Statistics\\
Carnegie Mellon University\\
\printead{a2}\\
\phantom{E-mail:\ }}

\address{Department of Mathematics\\
Harvard University\\
\printead{a3}\\
\phantom{E-mail:\ }}
\end{aug}

\begin{abstract}
We describe the asymptotic properties of the edge-triangle
exponential random graph model as the natural parameters diverge
along straight lines. We show that as we continuously vary the
slopes of these lines, a typical graph drawn from this model
exhibits quantized behavior, jumping from one complete
multipartite graph to another, and the jumps happen precisely at
the normal lines of a polyhedral set with infinitely many facets.
As a result, we provide a complete description of all asymptotic
extremal behaviors of the model.
\end{abstract}



\end{frontmatter}

\section{Introduction}
\label{intro}


Over the last decades, the availability and widespread diffusion
of network data on typically very large scales have created the
impetus for the
development of 
 new theories and methods for the analysis of large random graphs.
Despite the vast and rapidly growing body of literature on
network analysis (see, e.g., \cite{F1,F2,GXFA,Kol,Newman.book} and
references therein),  the study of the asymptotic behavior of
network models has proven rather difficult in most cases.
As a result,
methodologies for carrying out basic statistical
tasks such as parameter estimation, hypothesis and goodness-of-fit
testing with provable asymptotic guarantees have yet to be
developed for most network models.


Exponential random graph models \cite{FS,HL,WF} form one of the
most prominent
class of statistical models for random graphs, but also one for
which the issue of lack of understanding of their general
asymptotic properties is particularly pressing. These rather
generic models  are exponential families of probability
distributions over graphs, whereby the natural sufficient
statistics are virtually any functions on the space of graphs that
are deemed to capture essential features
 of interest. Such statistics may include, for instance, the number of edges
or copies of any finite subgraph, as well as more complex
quantities such as the degree distribution, and
combinations thereof. Exponential random graph models are especially useful when one wants to
construct models that resemble observed networks, but without specifying an
explicit network formation mechanism. They are among the most
widespread models in network analysis, with  numerous applications
in the social sciences, statistics,
statistical mechanics, economics and other disciplines. See, e.g.,
\cite{ERGM4,ERGM1,ERGM2,S,ERGM3}.

Despite being exponential families with finite support,
the large scale properties of exponential random graph models
are neither simple nor standard.
In fact, for
random graph models which do not assume independent edges, very
little was known about their asymptotics (but see \cite{B} and
\cite{HJ})  until the work of Chatterjee and
Diaconis \cite{CD}. By combining the recent theory of graphons (see, e.g.,
\cite{L2})
with a deep result about large deviations for the Erd\H{o}s-R\'{e}nyi model
established by Chatterjee and Varadhan \cite{CV}, they showed that the limiting properties
of many exponential random graph models can be obtained by solving
a certain variational problem in the graphon space (see Section
\ref{graphon} for a summary of these important results). Such
a framework  provides a principled way of resolving the large
sample behavior of exponential random graph models and has in fact
led to novel and insightful findings about these models. Still, the variational technique in \cite{CD} often does
not admit an explicit solution and additional work is required.

In this article we advance our understating of the
asymptotics of exponential random graph models by giving a
complete characterization of the asymptotic extremal properties of
a simple yet challenging $2$-parameter exponential random graph
model. In detail, for $n \geq 2$, let $\mathcal{G}_n$ denote
the set of all simple (i.e., undirected, with no loops or
multiple edges) labeled graphs on $n$ nodes. Notice that $|\mathcal{G}_n|
= 2^{ n \choose 2}$. For a graph $G_n \in \mathcal{G}_n$
and a simple labeled graph $H$ with vertex set $V(H)$ such that $|V(H)| \leq n$, the density homomorphism of
$H$ in $G_n$ is
\begin{equation}
\label{t} t(H, G_n)=\frac{|\text{hom}(H,
G_n)|}{n^{|V(H)|}},
\end{equation}
where $|\text{hom}(H, G_n)|$ denotes the number of homomorphisms
from $H$ into $G_n$,  i.e., edge preserving maps from $V(H)$ to
$V(G_n)$. Thus $t(H,G_n)$ is the probability that a random
mapping from $V(H)$ into $V(G_n)$ is edge-preserving.
For each $n$, we consider the exponential family $\{
\PR_{n,\beta}, \beta \in \mathbb{R}^2\}$ of probability distributions
on $\mathcal{G}_n$ which assigns to a graph $G_n \in
\mathcal{G}_n$ the probability
\begin{equation}
    \label{pmf} \PR_{n,\beta}(G_n)=\exp\left(n^2(\beta_1
t(H_1,G_n)+\beta_2 t(H_2,G_n)-\psi_n(\beta))\right),
\end{equation}
where $\beta=(\beta_1,\beta_2)$ are tuning parameters, $H_1 = K_2$ is a single edge, $H_2$ is a pre-chosen finite
simple graph (say a triangle, a two-star, etc.), and
$\psi_n(\beta)$ is the normalizing constant satisfying
\begin{equation}
\label{psi} \exp\left(n^2 \psi_n(\beta)\right)=\sum_{G_n \in
\mathcal{G}_n} \exp\left(n^2 \left(\beta_1 t(H_1,G_n)+\beta_2
t(H_2,G_n)\right)\right).
\end{equation}
%
In statistical physics, we refer to $\beta_1$ as the particle parameter and $\beta_2$ as the
energy parameter \cite{Ma, RS}. Correspondingly, the exponential model (\ref{pmf}) is said to be ``attractive'' if
$\beta_2$ is positive and ``repulsive'' if $\beta_2$ is negative. Although seemingly simple, this model is well known for its
wealth of non-trivial features (see, e.g., \cite{H:03,R}) and
challenging asymptotics (see \cite{CD}).

A natural question to ask is
how different values of the parameters $\beta_1$ and $\beta_2$
would impact the global structure of a typical random graph $G_n$
drawn from \eqref{pmf} for large $n$. We will generalize the extremal results of Chatterjee and Diaconis \cite{CD} and
complete an exhaustive study of all the extremal properties
 of \eqref{pmf}
 when $H_2 = K_3$, i.e., when $H_2$ is a triangle. Identifying the extremal properties of the edge-triangle model is
not only interesting from a mathematical point of view, but
also provides new insights into the expressive
power of the model itself. Towards that end, we will generalize the {\it
double asymptotic} framework of \cite{CD} and consider two limit processes: the network size $n$ grows unbounded and the
 natural  parameters $\beta$ diverge along generic {\it straight} lines.
In our analysis we will elucidate the relationship between all
possible directions along which the natural parameters can diverge
and the way the model tends to place most of its mass on graph
configurations that resemble complete multipartite graphs for
large enough $n$. As it turns out, looking just at straight lines
is precisely what is needed to categorize all extremal behaviors
of the model. Especially, when $n$ is large and $\beta_2$ is large negative, the edge-triangle model is used in the modeling of the crystalline structure of solids near the energy ground state. As we continuously vary the slopes of these generic lines, a progressive transition through finer and finer multipartite structures is revealed. We summarize our contributions as follows.

First, we extend the variational analysis technique of \cite{CD}
         to show that the set of all extremal (in $\beta$) distributions of the
         edge-triangle model consists of
         degenerate distributions on all Tur\'{a}n
         graphons when taking the size of the network $n$ as infinity. We further exhibit a partition of all the possible half-lines or
         directions
         in $\mathbb{R}^2$ in the form of a collection of cones with apexes
         at the origin and disjoint interiors, whereby two
          sequences of natural parameters $\beta$ diverging along different
         half-lines in the same cone yield the same asymptotic extremal
         behavior. We refer to this result as an asymptotic quantization of the
         parameter space. Finally, we
         identify a countable set of critical directions along
         which the extremal behaviors of the edge-triangle model cannot be
         resolved.

We then present a
         different
         technique of analysis that relies on the notion of closure of
         exponential families \cite{BRN:78}.  In this approach, the extremal
         properties of the model correspond to
         its asymptotic (in $n$) boundary in the total variation
         topology. The main advantage of this method
         is its ability to resolve the model also along critical directions.
         Specifically, we will demonstrate that, along each such direction, as $n$ grows, the model becomes discontinuous
        in the natural parametrization by $\beta$, and describe explicitly the points of
        discontinuity. We remark that this phenomenon is asymptotic: for finite $n$ the
         natural parametrization by $\beta$
         is always continuous, even on the boundary of the total variation
         closure of the model. Unlike variational techniques, which characterize
         the properties of the
         model as a function of the parameter values $\beta$ when
         the network size $n$ is infinite, the approach based on the total
         variation closure considers $n$ finite (but increasing)
         and lets $\beta$ tend to infinity appropriately for each fixed $n$.

A central ingredient of our analysis is the use of simple yet effective geometric
         arguments that combine recent results in asymptotic extremal
	 graph theory \cite{Razborov} with the theory of graphons \cite{L2} and
	 the traditional theory of exponential families. Both the quantization of the parameter space and the
         identification of critical directions stem from the dual geometric
         property of a bounded convex polygon with infinitely many edges,
         which can be thought of as an asymptotic mean value
         parametrization of the edge-triangle exponential model. We expect this
         framework to apply more generally to other exponential random graph
         models.

The rest of this paper is organized as follows. In Section
\ref{limit} we provide some basics of graph limit theory,
summarize the main results of \cite{CD} and introduce key
geometric quantities. In Section \ref{attractive} we investigate
the asymptotic behavior of ``attractive'' 2-parameter exponential
random graph models along general straight lines. In Section
\ref{vertical} we analyze the asymptotic structure of
``repulsive'' 2-parameter exponential random graph models along
vertical lines. In Sections \ref{jump} and \ref{ale} we examine
the asymptotic feature of the edge-triangle model along general
straight lines. Section \ref{illu} shows some illustrative figures and Section \ref{further} is devoted to further discussions. All the proofs are in the appendix.

\section{Background}
\label{limit}

Below we will provide some background on the theory of graph
limits and its use in exponential random graph models, focusing in
particular on the edge-triangle model.

\subsection{Graph limit theory and graph limits of exponential random graph models}
\label{graphon}
A series of recent important contributions by mathematicians
and physicists 
have led to a unified and elegant theory of limits of sequences of
dense graphs. See, e.g.,  \cite{BCLSV1, BCLSV2,
BCLSV3, L1, LS} and the book \cite{L2} for a comprehensive account
and references. See also the related work on exchangeable arrays,
where some of these results had already been derived:
\cite{Aldous, DJ:08, Hoover, Kallenberg:2005, Lauritzen:2003}.


Here are the basics of this theory. A sequence $\{ G_n
\}_{n=1,2,\ldots}$ of graphs, where we assume $G_n
\in\mathcal{G}_n$ for each $n$, is said to converge when, for
every simple graph $H$, $\lim_{n \to \infty} t(H, G_n)=t(H)$
    for some $t(H)$. The main result in \cite{LS} is a complete characterization
    of all limits of converging graph sequences, which are shown to correspond to the
    functional space $\mathcal{W}$ of all symmetric measurable functions from  $[0,1]^2$
    into $[0,1]$, called {\it graph limits} or {\it graphons}. Specifically, the
    graph sequence $\{
    G_n\}_{n=1,2,\ldots}$ converges
    if and only if there exists a graphon $f \in \mathcal{W}$ such that, for every
    simple graph $H$ with vertex set $\{1,\ldots,k\}$  and edge set $E(H)$,
    \begin{equation}
    \label{tt}
    \lim_{n \to \infty} t(H, G_n) = t(H, f) :=\int_{[0,1]^{k}}\prod_{\{i,j\}\in E(H)}f(x_i, x_j)dx_1\cdots
dx_k.
    \end{equation}
Any finite
graph $G_n$ can be represented as a graphon of the form
\begin{equation}
\label{f}
f^{G_n}(x,y)=\left\{%
\begin{array}{ll}
    1, & \hbox{if $(\lceil nx \rceil, \lceil ny \rceil)$ is an edge in $G_n$,} \\
    0, & \hbox{otherwise,} \\
\end{array}%
\right.
\end{equation}
where $\lceil x \rceil$ denotes
    the smallest integer no less than $x \in \mathbb{R}$.
Among the main advantages of the graphon framework is its ability
to represent the limiting properties of sequences of graphs $G_n$, which
are discrete objects that lie in different probability spaces, with the unified functional space $\mathcal{W}$.
Lov\'{a}sz and
Szegedy \cite{LS} showed that convergence of all graph
homomorphism densities is equivalent to a certain cut metric convergence in the quotient graphon space
$(\widetilde{\mathcal{W}}, \delta_{\square})$, which is obtained after taking into account measure preserving transformations. A sequence of
(possibly random) graph $\{G_n\}_{n=1,2,\ldots}$ converges to a
graphon $f$ if and only if $\delta_{\square}(\tilde{f}^{G_n},\tilde{f})\to 0\hbox{
in
  probability as } n\to \infty$,
    where $f^{G_n}$ is defined in \eqref{f}.
It may be worth emphasizing that
graphons described here are tailored to limits of {\it dense} graphs, i.e., graphs
having order $n^2$ edges. In particular, they cannot discern any graph property in the sequence that depends on a
number of edges of order $o(n^2)$.

In a recent important paper, Chatterjee and Diaconis \cite{CD} utilized the nice analytic properties of the
metric space $(\widetilde{\mathcal{W}}, \delta_{\square})$ and examined
the asymptotic behavior of exponential random graph models. For
the purpose of this paper, two results from \cite{CD} are
particularly significant. The first result, which is an
application of a deep large deviations result of \cite{CV}, is
Theorem 3.1 in \cite{CD}. When applied to the $2$-parameter
exponential random graph models mentioned above it implies that
the limiting normalizing constant $\psi_\infty(\beta)=\lim_{n\to
\infty}\psi_n(\beta)$
 always exists and is given by
\begin{equation}
\label{setmax} \psi_\infty(\beta)=\sup_{\tilde{f}\in
\tilde{\mathcal{W}}}\left(\beta_1 t(H_1,f)+\beta_2 t(H_2,f)-
\iint_{[0,1]^2}I(f)dxdy\right),
\end{equation}
where $f$ is any representative element of the equivalence class
$\tilde{f}$, and
\begin{equation}
I(u)=\frac{1}{2}u\log u+\frac{1}{2}(1-u)\log(1-u).
\end{equation}
The second result, Theorem 3.2 in \cite{CD}, is concerned with the
solutions of the above variational optimization problem. In
detail, let $\tilde{F}^{*}(\beta)$ be the subset of
$\widetilde{\mathcal{W}}$ where (\ref{setmax}) is maximized. Then,
the quotient image $\tilde{f}^{G_n}$ of a random graph $G_n$ drawn
from (\ref{pmf}) must lie close to $\tilde{F}^{*}(\beta)$ with
probability vanishing in $n$, i.e.,
\begin{equation}\label{eq:32CD}
\delta_{\square}(\tilde{f}^{G_n}, \tilde{F}^*(\beta))\to 0\hbox{
in
  probability as } n\to \infty.
\end{equation}


Due to its complicated structure, the variational problem
(\ref{setmax}) is not always explicitly solvable. So far major
simplification has only been achieved when $\beta_2$ is positive
or negative with small magnitude. For $\beta_2$ lying in
these parameter regions, Chatterjee and Diaconis \cite{CD} showed that $G_n$ behaves like an Erd\H{o}s-R\'{e}nyi
graph $G(n, u)$ in the large $n$ limit, where $u$ is picked
randomly from the set $U$ of maximizers of a reduced form of
(\ref{setmax}):
\begin{equation}
\label{lmax} \psi_{\infty}(\beta)=\sup_{0\leq u\leq
1}\left(\beta_1 u^{e(H_1)}+\beta_2 u^{e(H_2)}-I(u)\right),
\end{equation}
where $e(H_i)$ is the number of edges in $H_i$. (There are also
related results in H\"{a}ggstr\"{o}m and Jonasson \cite{HJ} and
Bhamidi et al. \cite{B}.) Chatterjee and Diaconis \cite{CD} also
studied the case in which $H_1 = K_2$ and $H_2$ is arbitrary,
$\beta_1$ is fixed and $\beta_2 \rightarrow -\infty$, and showed that a typical graph
$G_n$ from \eqref{pmf} will be close
to a random subgraph of a complete multipartite graph with the
number of classes depending on the chromatic number of $H_2$ (see
Section \ref{vertical} for the exact statement of this result).

\subsection{Edge-triangle exponential random graph model and its asymptotic
geometry} \label{extreme} In this article we focus almost exclusively on the
edge-triangle model, which is the exponential random graph model
obtained by setting in \eqref{pmf} $H_1 = K_2$ and $H_2 = K_3$.
Explicitly, in the edge-triangle model the probability of a graph
$G_n \in \mathcal{G}_n$ is
\begin{equation}
    \label{eq:edge-triangle} \PR_{n,\beta}(G_n)=\exp\left(n^2(\beta_1
t(K_2,G_n)+\beta_2 t(K_3,G_n)-\psi_n(\beta))\right),
\end{equation}
where $\psi_n(\beta)$ is given in \eqref{psi} and there are no restrictions on how the natural parameters $\beta$ diverge. Below we describe
the asymptotic geometry of this model, which underpins much of our
analysis.

To start off, for any $G_n \in \mathcal{G}_n$, the vector of the
densities of graph homomorphisms of $K_2$ and $K_3$ in $G_n$ takes
the form
\begin{equation}\label{eq:tGn}
t(G_n) =  \left(
\begin{array}{c}
    t(K_2,G_n)\\
    t(K_3,G_n)\\
\end{array}
\right) =  \left(
\begin{array}{c}
    \frac{2 E(G_n)}{n^2}\\
    \frac{6 T(G_n)}{n^3}
\end{array}
\right) \in [0,1]^2,
\end{equation}
    where $E(G_n)$ and $T(G_n)$ are the number of subgraphs of $G_n$ isomorphic
    to $K_2$  and $K_3$, respectively.
    Since every finite graph can be represented as a graphon, we can extend
    $t$ to a map from $\mathcal{W}$ into $[0,1]^2$ by setting (see (\ref{tt}))
\begin{equation}
\label{th} t(f) = \left(
\begin{array}{c}
    t(K_2,f)\\
    t(K_3,f)
\end{array}
\right),\quad f \in \mathcal{W}.
\end{equation}

As we will see, the asymptotic extremal behaviors of the
edge-triangle
   model can  be fully characterized by the geometry of two compact subsets of
   $[0,1]^2$.
   The first is the set
   \begin{equation}\label{eq:R}
    R = \{ t(f), f \in \mathcal{W}\}
   \end{equation}
    of all realizable values of the edge and triangle density homomorphisms as $f$
    varies over $\mathcal{W}$. The second set, $P$, is  the convex hull of $R$,
    i.e.,
    \begin{equation}
        \label{eq:P}
        P = \mathrm{convhull}(R).
    \end{equation}
    Figures \ref{et} and \ref{p} depict $R$ and $P$, respectively.

To describe the properties of the sets $R$ and $P$, we
introduce some quantities
 that we will use throughout this paper.
   For $k=0,1,\ldots$,  we set $v_k =
    t(f^{K_{k+1}})$, where $f^{K_1}$ is the identically zero graphon and,  for
    any integer $k>1$,
    \begin{equation}
    \label{eq:turan-graphon}
f^{K_{k}}(x,y) = \left\{%
\begin{array}{ll}
    1 & \hbox{if $\lceil x k \rceil \neq  \lceil y k \rceil$,} \\
    0 & \hbox{otherwise,} \\
\end{array}%
\right. \quad (x,y) \in [0,1]^2
\end{equation}
    is the Tur\'{a}n graphon with $k$ classes. Thus,
\begin{equation}
    \label{eq:vk}
       v_k = \left(
       \begin{array}{c}
       \frac{k}{k+1}\\
\frac{k(k-1)}{(k+1)^2}
\end{array}\right), \quad k=0,1,\ldots.
\end{equation}
Note that any graphon $f$ with $t(f)=v_k$ is equivalent to the Tur\'{a}n graphon $f^{K_{k+1}}$.
The name Tur\'{a}n graphon is due to the easily verified fact that
\[
    v_k = \lim_{n \rightarrow \infty} v_{k,n}, \quad \text{for each }  k=1,2,\ldots,
    \]
    with $v_{k,n} = t(T(n,k+1))$, the homomorphism densities of $K_2$ and $K_3$
    in $T(n,k+1)$, i.e., a
    Tur\'{a}n graph on $n$ nodes with $k+1$ classes. Tur\'{a}n graphs are well known to provide the solutions of many extremal dense graph
    problems (see, e.g., \cite{Diestel}), and will turn out to be the extremal graphs for
    the edge-triangle model as well.
		
    The set $R$ is a classic and well studied object in asymptotic extremal
    graph theory,
    even though the precise shape of its boundary was
    determined only recently  (see, e.g., \cite{Bo, Fisher, Goodman, LS1} and
the book
    \cite{L2}). Letting $e$ and $t$ denote the coordinate corresponding to the
    edge and triangle density homomorphisms, respectively, the upper boundary curve of $R$ (see Figure \ref{et}), is given by the equation $t =
    e^{3/2}$, and can be derived using the Kruskal-Katona theorem (see Section
    16.3 of \cite{L2}).
The lower boundary curve is trickier. The trivial lower bound of
$t = 0$, corresponding to the horizontal segment, is attainable at
any $0 \leq e \leq 1/2$ by graphons describing the (possibly
asymptotic) edge density of subgraphs of complete bipartite
graphs.
 For $e \geq 1/2$, the optimal
bound was obtained recently by Razborov \cite{Razborov}, who
established, using the flag algebra calculus, that for $(k-1)/k
\leq e \leq k/(k+1)$ with $k \geq 2$,
\begin{equation}
\label{Ra} t \geq \frac{(k-1) \left(k-2\sqrt{k(k-e(k+1))}\right)
\left(k+ \sqrt{k(k-e(k+1))} \right)^2}{k^2 (k+1)^2}.
\end{equation}
All the curve segments describing the lower boundary of $R$ are
easily seen to be strictly convex, and the boundary points of
these segments are precisely the points $v_k$, $k=0,1,\ldots$.

    The following Lemma \ref{lem:P}
    is a direct consequence of Theorem 16.8 in \cite{L2} (see page 287 of the
    same reference for details). Below,  $\mathrm{cl}(A)$ denotes the
    topological closure of
    the set $A \subset \mathbb{R}^2$.

    \begin{lemma}\label{lem:P}
    \begin{enumerate}
    \item $R = \mathrm{cl} \big( \left\{ (t(G_n), G_n \in \mathcal{G}_n, n=1,2,\ldots
    \right\} \big).$
    \item The extreme points of $P$ are the points $ \{ v_k, k=0,1,\ldots\}$
        and the point $(1,1) = \lim_{k \rightarrow \infty} v_k$.
    \end{enumerate}
\end{lemma}


The first result of Lemma \ref{lem:P} indicates that the set of
edge and triangle homomorphism densities of all finite graphs is
dense in $R$. The second result implies that the boundary of $P$
consists of infinitely many segments with endpoints $v_k$, for
$k=0,1,\ldots$, as well as the line segment joining $v_0 = (0,0)$
and $(1,1) = \lim_{k \rightarrow \infty} v_k$. 
For $k=0,1,\ldots$, let $L_k$ be the segment joining the adjacent
vertices $v_k$ and $v_{k+1}$, and $L_{-1}$ the segment joining
$v_0$ and the point $(1,1)$. Each such $L_k$ is an exposed face of
$P$ of maximal dimension $1$, i.e., a facet. Notice that the
length of the segment $L_k$ decreases monotonically to zero as $k$
gets larger. For any $k>0$, the slope of the line passing through
$L_k$ is
\[\frac{k(3k+5)}{(k+1)(k+2)},
\]
    which increases monotonically to $3$ as $k\rightarrow \infty$.
Simple algebra yields that the facet $L_k$ is exposed by the
vector
    \begin{equation}\label{eq:outer}
    o_k =     \left\{
            \begin{array}{ll}
        (-1,1) & \text{if } k=-1,\\
        (0,-1) & \text{if } k=0,\\
        \left( 1, - \frac{(k+1)(k+2)}{k(3k+5)} \right) & \text{if }k =
            1,2,\ldots.\\
        \end{array}
        \right.
     \end{equation}

The vectors $o_k$ will play a key role in determining the
asymptotic behavior of the edge-triangle model, so much so that
they deserve their own names.

\begin{definition}
     The vectors $\{o_k, k=-1,0,1,\ldots \}$ are the {\it critical
     directions} of the edge-triangle model.
\end{definition}

\begin{figure}
\centering
\includegraphics[width=3in]{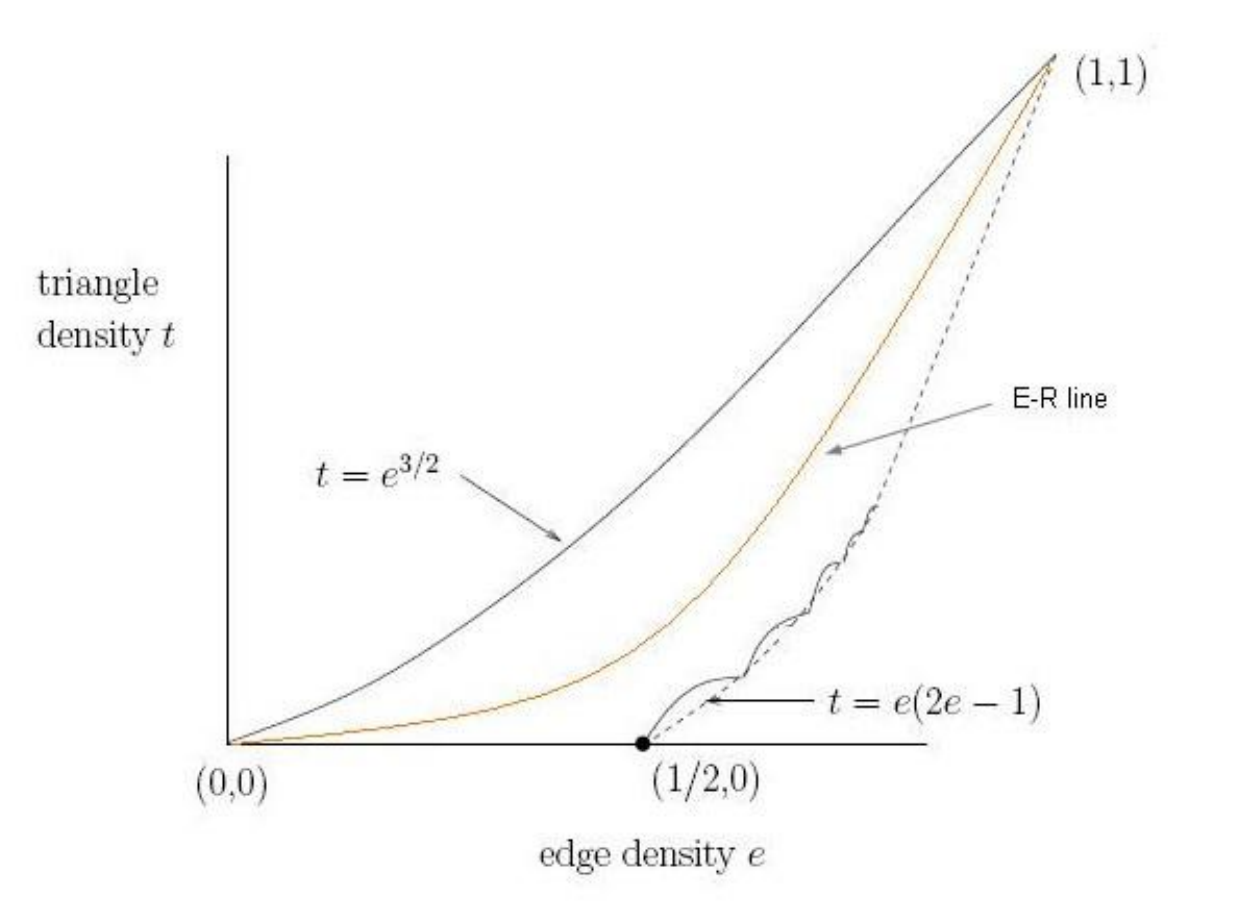}
\caption{The set $R$ of all feasible edge-triangle homomorphism
densities,
    defined in \eqref{eq:R}.} \label{et}
\end{figure}

\begin{figure}
\centering
\includegraphics[width=3in]{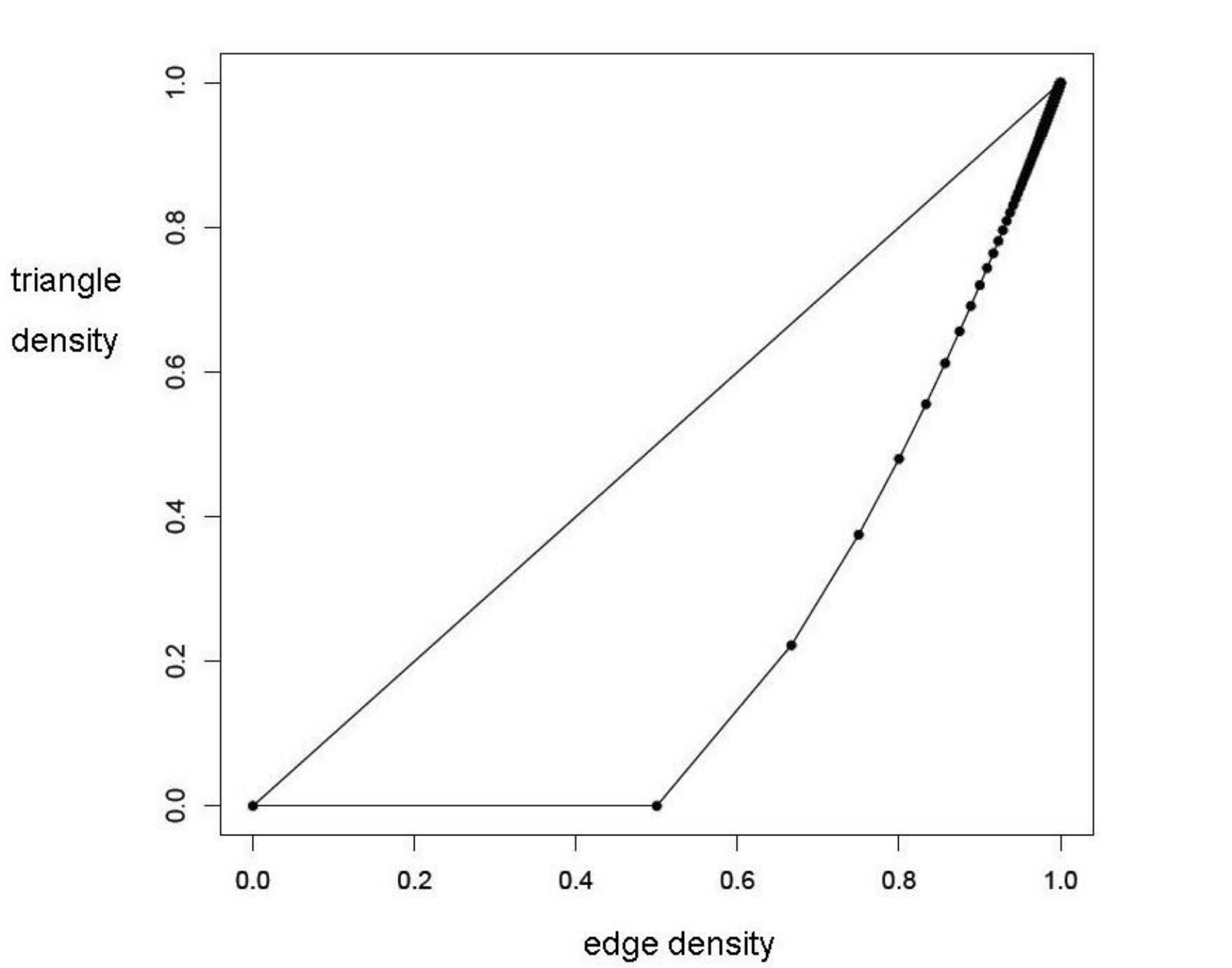}
\caption{The set $P$ described in \eqref{eq:P}.} \label{p}
\end{figure}

     For a set $A \subset \mathbb{R}^2$, define $\mathrm{cone}(A)$ as the set
      of all conic combinations of points in $A$. It follows that the outer normals to the facets of $P$ are given by
      \[
           \mathrm{cone}(o_k), \quad k=-1,0,1,\ldots,
      \]
      i.e., by rays in $\mathbb{R}^2$
  emanating from the origin and going along the
direction of $o_k$. Finally, for $k=0,1,\ldots$ let  $C_k =
\mathrm{cone}(o_{k-1},o_k)$ denote the normal cone to $P$ at
$v_k$, i.e., a 2-dimensional pointed polyhedral cone spanned by $o_{k-1}$ and
$o_k$. Denote by $C_k^{\circ}$ the topological interior of $C_k$. Then, since $P$ is bounded, for any non-zero $x \in
\mathbb{R}^2$, there exists one $k$ for which either $x \in
\mathrm{cone}(o_k)$ or $x \in C_k^{\circ}$. The normal cones to the faces of $P$ form a locally finite
polyhedral complex of cones, shown in Figure \ref{dir}. As our
results will demonstrate, each one of these cones uniquely
identifies one of infinitely many asymptotic extremal behaviors of
the edge-triangle model.

\section{Variational analysis}
\label{mei}

In this section we characterize the extremal properties of
$2$-parameter exponential random graphs and especially of the
edge-triangle model using the variational approach described in
Section \ref{graphon}. Chatterjee and Diaconis \cite{CD} showed
that a typical graph drawn from a 2-parameter exponential random
graph model with $H_1$ an edge and $H_2$ a fixed graph with
chromatic number $\chi$  is a $(\chi-1)$-equipartite graph when
$n$ is large, $\beta_1$ is fixed, and $\beta_2$ is large and
negative, i.e., when the two parameters trace a vertical line
downward.

In the hope of discovering other interesting extremal behaviors,
we investigate the asymptotic structure of 2-parameter exponential
random graph models along general straight lines. In particular,
we will study sequences of model parameters of the form $\beta_1 =
a \beta_2 + b$, where $a$ and $b$ are constants and $|\beta_2|
\rightarrow \infty$. Thus, for any $\beta = (\beta_1,\beta_2) \in
\mathbb{R}^2$, we can regard the quantities $\tilde{F}^*(\beta)$
and $\psi_\infty(\beta)$, defined in Section \ref{graphon},  as
functions of $\beta_2$ only, and therefore will write them as
$\tilde{F}^*(\beta_2)$ and $\psi_\infty(\beta_2)$ instead.

While we only give partial results for general exponential random
graphs, we are able to provide a nearly complete characterization
of the edge-triangle model. Even more refined results are
possible, as will be shown in Section \ref{ale}.

\subsection{Asymptotic behavior of attractive $2$-parameter exponential random
graph models along general lines} \label{attractive} We will first
consider the asymptotic behavior of ``attractive'' $2$-parameter
exponential random graph models as $\beta_2 \rightarrow \infty$.
We will show that for $H_1$ an edge and $H_2$ any other finite
simple graph, in the large $n$ limit, a typical graph drawn from
the exponential model becomes complete under the topology induced
by the cut distance if $a>-1$ or $a=-1$ and $b>0$; it becomes
empty if $a<-1$ or $a=-1$ and $b<0$; while for $a=-1$ and $b=0$,
it either looks like a complete graph or an empty graph. Below,
for a non-negative constant $c$, we will write $u=c$ when $u$ is
the constant graphon with value $c$.

\begin{theorem}
\label{thm.at} Consider the 2-parameter exponential random graph
model \eqref{pmf}, with  $H_1 = K_2$ and $H_2$ a different,
arbitrary graph. Let $\beta_1=a\beta_2+b$. Then
\begin{equation}
    \lim_{\beta_2 \rightarrow \infty} \sup_{\tilde{f} \in \tilde{F}^*(\beta_2) }
    \delta_{\square} (\tilde{f}, \tilde{U})=0,
\end{equation}
where the set $U \subset \mathcal{W}$ is determined as follows:
\begin{itemize}
    \item $U=\{1\}$ if $a>-1$ or $a=-1$ and $b>0$,
    \item $U=\{0, 1\}$ if
$a=-1$ and $b=0$, and \item $U=\{0\}$ if $a<-1$ or $a=-1$ and
$b<0$.
\end{itemize}
\end{theorem}

    When $a=-1$ and $b=0$, the limit points of the solution set of the variational problem
    \eqref{lmax} consist of two radically different graphons, one specifying
    an asymptotic edge density of $1$ and the other
    of $0$. This intriguing behavior was captured in \cite{RY},
    where it was shown that there is a continuous curve that
    asymptotically approaches the line $\beta_1=-\beta_2$, across
    which the graph transitions from being very sparse to very dense. Unfortunately, the variational technique used in
the proof of the theorem does not
    seem to yield a way of deciding whether only one or both solutions can actually be
    realized. As we will see next, a similar issue arises
    when analyzing the asymptotic extremal behavior of the edge-triangle model
    along critical directions (see Theorem \ref{main}).
    In Section \ref{ale} we describe a different method of analysis that will
    allow us to resolve this rather subtle ambiguity within the edge-triangle
    model and reveal an asymptotic phase transition
    phenomenon. In particular, Theorem \ref{thm:main.3} there can be easily adapted
    to provide an analogous resolution of the case $a=-1$ and
    $b=0$ in Theorem \ref{thm.at}.

We remark that, using same arguments, it is also possible to
handle the case in which $\beta_2$ is fixed and $\beta_1$ diverges
along horizontal lines. Then, we obtain the intuitively clear
result that, in the large $n$ limit, a typical random graph drawn
from this model becomes complete if $\beta_1 \rightarrow \infty$,
and empty if $\beta_1 \rightarrow -\infty$. We omit the easy
proof.

\begin{figure}
\centering
\includegraphics[width=2.5in]{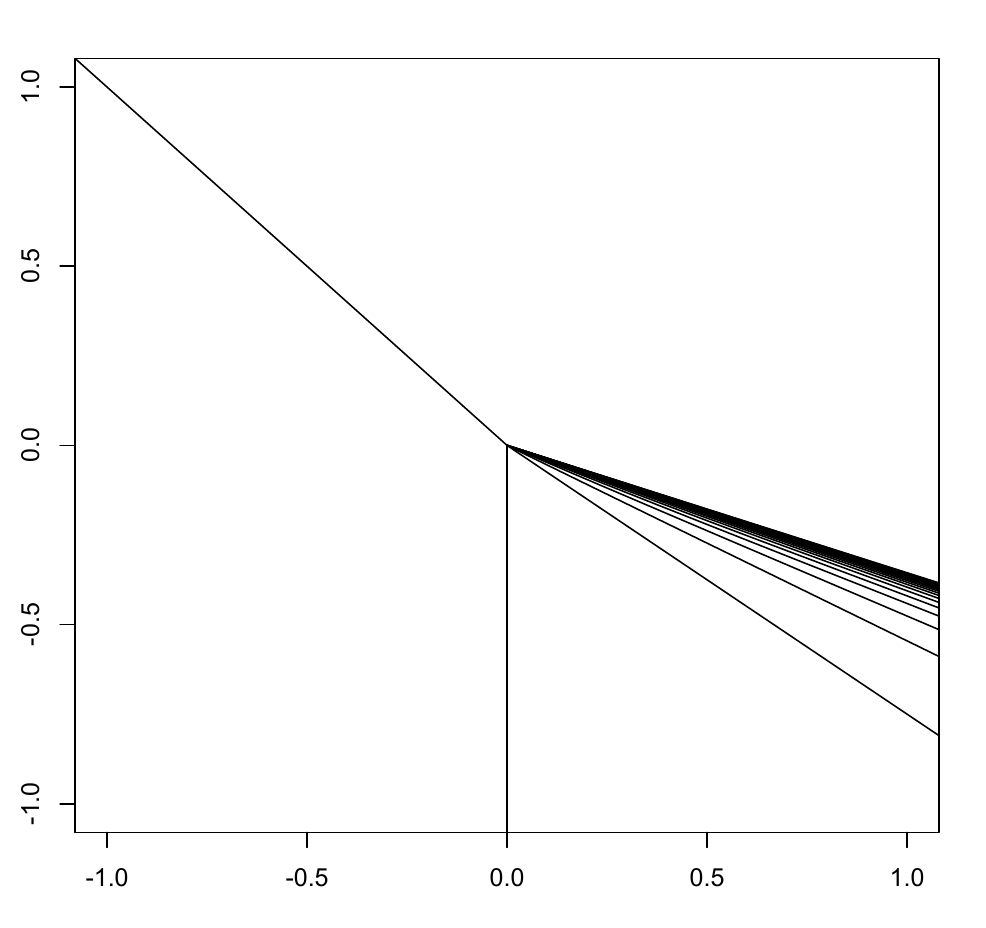}
\caption{Cones, i.e., rays emanating from the origin, generated by
the critical
    directions $o_{-1}$, $o_0$ and $o_k$, for $k=1,\ldots, 40$.
 The plot also provides an accurate depiction of the locally finite polyhedral
 complex comprised  by the normal cones to the faces of the set $P$ defined
 in \eqref{eq:P}.} \label{dir}
\end{figure}

The next two sections deal with the more challenging analysis of
the asymptotic behavior of ``repulsive'' $2$-parameter exponential
models as $\beta_2 \rightarrow -\infty$. As mentioned earlier, the asymptotic properties of such models are largely
unknown in this region.

\subsection{Asymptotic behavior of repulsive $2$-parameter exponential random graph models along vertical lines}
\label{vertical} The purpose of this section is to give an
alternate proof of Theorem 7.1 in \cite{CD} that uses
classic results in extremal graph theory. In addition, this
general result covers the asymptotic extremal behavior of the
edge-triangle model along the vertical critical direction.

Recall that $\beta_1$ is fixed and we are interested in  the
asymptotics of $\tilde{F}^*(\beta_2)$ and $\psi_\infty(\beta_2)$
as $\beta_2 \rightarrow -\infty$. We point out here that the limit
process in $\beta_2$ may also be interpreted by taking
$\beta_1=a\beta_2+b$ with $a=0$ and $b$ large negative. The
importance of this latter interpretation will become clear in the
next section. Our work here is inspired by related results of
Fadnavis \cite{F} and Radin and Sadun \cite{RS} in the case of $H_2$ being
a triangle.

\begin{theorem} [Chatterjee-Diaconis]
\label{ChDi} Consider the 2-parameter exponential random graph
model \eqref{pmf}, with $H_1 = K_2$ and $H_2$ a different,
arbitrary graph. Fix $\beta_1$. Let $r=\chi(H_2)$ be the chromatic
number of $H_2$. Let
$p=e^{2\beta_1}/(1+e^{2\beta_1})$. Then
\begin{equation}
\lim_{\beta_2 \rightarrow -\infty} \sup_{\tilde{f}\in
\tilde{F}^*(\beta_2)} \delta_{\square} (\tilde{f}, \tilde{U})=0,
\end{equation}
where the set $U \subset \mathcal{W}$ is given by $U=\{pf^{K_{r-1}}\}$ (see (\ref{eq:turan-graphon})).
\end{theorem}

As explained in \cite{CD}, the above result can be interpreted as
follows: if $\beta_2$ is negative and large in magnitude and $n$
is big, then a typical graph $G_n$ drawn from the 2-parameter
exponential model \eqref{pmf} looks roughly like a complete
$(\chi(H_2)-1)$-equipartite graph with $1-p$ fraction of edges
randomly deleted, where $p=e^{2\beta_1}/(1+e^{2\beta_1})$.

\subsection{Asymptotic quantization of edge-triangle model along general lines}
\label{jump} In this section we conduct a thorough analysis of the
asymptotic behavior of the edge-triangle model as $\beta_2
\rightarrow -\infty$. As usual, we take $\beta_1=a\beta_2+b$,
where $a$ and $b$ are fixed constants. The $a=0$ situation is a
special case of what has been discussed in the previous section:
If $n$ is large, then a typical graph $G_n$ drawn from the
edge-triangle model looks roughly like a complete bipartite graph
with $1/(1+e^{2b})$ fraction of edges randomly deleted. It is not
too difficult to establish that if $a>0$, then independent of $b$,
a typical graph $G_n$ becomes empty in the large $n$ limit.
Intuitively, this should be clear: $\beta_1$ and $\beta_2$ both
large and negative entail that $G_n$ would have minimal edge and
triangle densities. However, the case $a<0$ leads to an array of
non-trivial and intriguing extremal behaviors for the
edge-triangle model, and they are described in our next result. We
emphasize that our analysis relies on the
explicit characterization by Razborov \cite{Razborov} of the lower
boundary of the set $R$ of (the closure of) all edge and triangle
density homomorphisms (see (\ref{Ra})) and on the fact that the
extreme points of $P$ are the points $\{v_k, k=0,1,\ldots\}$,
given in \eqref{eq:vk}\footnote{
In fact, we do not actually need the
exact expressions of the lower boundary of homomorphism densities
(the Razborov curve) to derive our results -- all we need is its
strict concavity. According to Bollob\'{a}s \cite{B:76, Bo}, the
vertices of the convex hull $P_n$ for $K_2$ and $K_n$, not just
$K_2$ and $K_3$ as in the edge-triangle case, are given by the
limits of complete $k$-equipartite graphs. More general
conjectures of the limiting object $P$ may be found in
\cite{EN:11}.
}.
     Recall that these points correspond to the
density homomorphisms of the Tur\'{a}n graphons $f^{K_{k+1}}$,
$k=0,1,\ldots$, as shown in \eqref{eq:turan-graphon}.

\begin{theorem}
\label{main} Consider the edge-triangle exponential random graph
model (\ref{eq:edge-triangle}). Let $\beta_1=a\beta_2+b$ with
$a<0$ and, for $k\geq 0$, let $a_k=-\frac{k(3k+5)}{(k+1)(k+2)}$.
Then,
\begin{equation}
    \lim_{\beta_2 \rightarrow -\infty} \sup_{\tilde{f} \in \tilde{F}^*(\beta_2)}
    \delta_{\square}(\tilde{f}, \tilde{U})=0,
\end{equation}
where the set $U \subset \mathcal{W}$ is determined as follows:
\begin{itemize}
    \item $U= \{ f^{K_{k+2}} \}$ if $a_{k}>a>a_{k+1}$ or $a=a_k$ and $b>0$,
    \item $U=\{f^{K_{k+1}}, f^{K_{k+2}}\}$ if $a=a_k$ and $b=0$, and
    \item $U= \{ f^{K_{k+1}} \}$ if $a=a_k$ and $b<0$.
\end{itemize}
 \end{theorem}

\begin{remark}
Notice that the case $a=a_k$ and $b=0$ corresponds to the critical
direction $o_k$, for $k=1,2,\ldots$.
\end{remark}

The above result says that,  if $\beta_1=a\beta_2+b$ with $a<0$
and $\beta_2$ large negative, then in the large $n$ limit, any
graph drawn from the edge-triangle model is indistinguishable in
the cut metric topology from a complete $(k+2)$-equipartite graph
if $a_k>a>a_{k+1}$ or $a=a_k$ and $b>0$; it looks like a complete
$(k+1)$-equipartite graph if $a=a_k$ and $b<0$; and for $a=a_k$
and $b=0$, it either behaves like a complete $(k+1)$-equipartite
graph or a complete $(k+2)$-equipartite graph.
 Lastly it becomes complete if $a\leq \lim_{k \rightarrow \infty} a_k=-3$.
 Overall, these results describe in a precise manner the array of all possible
 asymptotic extremal behaviors of the
 edge-triangle model, and link  them directly to the geometry of the natural
 parameter space as captured by the polyhedral complex of cones shown in Figure
 \ref{dir}.

When $a=a_k$ and $b=0$, for any $k=0,1,\ldots$, i.e., when the
parameters diverge along the critical direction $o_k$, Theorem
\ref{main} suffers from the same ambiguity as  Theorem
\ref{thm.at}: the limit points of the solution set of the
variational problem \eqref{setmax} as $\beta_2 \rightarrow
-\infty$ are Tur\'{a}n graphons with $k+1$ and $k+2$ classes.
Though already quite informative, this result remains somewhat
unsatisfactory because it does not indicate whether both such
graphons are actually realizable in the limit and in what manner.
As we remarked in the discussion following Theorem \ref{thm.at},
our method of proof, largely based on and inspired by the results
in \cite{CD}, does not seem to suggest a way of clarifying this
issue. In the next section we will present a completely different
asymptotic analysis yielding different types of convergence guarantees. As in the present section, two limit processes will be considered: the network size $n$ grows unbounded and the natural parameters $\beta$ diverge, with the order of limits interchanged. The two approaches in Sections \ref{mei} and \ref{ale} produce similar results except along critical directions, where the second approach has the added power of resolving the aforementioned
ambiguities.

\subsection{Probabilistic convergence of sequences of graphs from
edge-triangle model}\label{sec:cut.to.cut} The results obtained in
Sections \ref{attractive}, \ref{vertical} and \ref{jump}
characterize the extremal asymptotic behavior of the edge-triangle
model through functional convergence in the cut topology within
the space $\widetilde{\mathcal{W}}$. Our explanation of such
results though has more of a probabilistic flavor. Here we briefly
show how this interpretation is justified. By combining
\eqref{eq:32CD}, established in Theorem 3.2 of \cite{CD}, with the
theorems in Sections \ref{attractive}, \ref{vertical} and
\ref{jump}, and a standard diagonal argument, we can deduce the
existence of subsequences of the form $\{ (n_i,\beta_{2,i}
)\}_{i=1,2\ldots}$, where $n_i \rightarrow \infty$ and
$\beta_{2,i} \rightarrow \infty$ or $-\infty$ as $i \rightarrow
\infty$, such that the following holds. For fixed $a$ and $b$, let
$\{ G_i \}_{i=1,2,\ldots}$ be a sequence of random graphs drawn
from   the sequence of edge-triangle models with node sizes $\{
n_i \}$ and parameter values $\{ (a
    \beta_{2,i} + b, \beta_{2,i}) \}$. Then
\[
    \delta_{\square}(\tilde{f}^{G_i},
    \tilde{U}) \rightarrow 0\hbox{
in
  probability as } i\rightarrow
    \infty,
    \]
where the set $U \subset \mathcal{W}$, which depends on $a$ and
$b$, is described in Theorems \ref{thm.at}, \ref{ChDi} and
\ref{main}. In Section \ref{sec:var.to.cut} we will obtain a very
similar result by entirely different means.

\section{Finite $n$ analysis}
\label{ale}  In the remainder part of this paper we will present
an alternative analysis of the asymptotic behavior of the
edge-triangle model using directly the properties of the
exponential families and their closure in the finite $n$ case
instead of the variational approach of \cite{CD, RS, RY}. Though
the results in this section are seemingly similar to the ones in
Section \ref{mei}, we point out that there are marked differences.
First, while in Section \ref{mei} we study convergence in the cut
metric for the quotient space $\widetilde{\mathcal{W}}$, here we
are concerned instead with convergence in total variation of the
edge and triangle homomorphism densities. Secondly, the double
asymptotics, in $n$ and in the magnitude of $\beta$, are not the
same.
 In Section \ref{mei},
the system size $n$ goes to infinity first followed by the
divergence of the parameter $\beta_2$ to positive or negative
infinity.
In contrast, here we first let the magnitude of the natural
parameter $\beta$ diverge to infinity so as to isolate a simpler
``restricted" edge-triangle sub-model,  and then study its limiting
properties as $n$ grows. Though both approaches are in fact
asymptotic, we characterize the latter  as ``finite $n$", to
highlight the fact that we are not working with a limiting system
and because, even with finite $n$, the extremal properties already
begin to emerge. Despite these differences, the conclusions we can
derive from both types of analysis are rather similar.
Furthermore, they imply a nearly identical convergence in
probability in the cut topology (see Sections \ref{sec:cut.to.cut}
and \ref{sec:var.to.cut}).

Besides giving a rather strong form of asymptotic convergence, one
of the appeals of the finite $n$ analysis consists in its ability
to provide a more detailed categorization of the limiting behavior
of the model along critical directions using simple geometric
arguments based on the dual geometry of $P$, the convex hull of
edge-triangle homomorphism densities. Specifically, we will
demonstrate that, asymptotically, the edge-triangle model
undergoes phase transitions along critical directions, where its
homomorphism densities will converge in total variation
 to the densities of one of two Tur\'{a}n graphons, both of which are
 realizable. In addition we are able to state precise conditions on the natural
 parameters for such transitions to
 occur.

\subsection{Exponential families}
We begin by reviewing some of the standard theory of exponential
families and their closure in the context of the edge-triangle
model. We refer the readers to Barndorff-Nielsen \cite{BRN:78} and
Brown \cite{BROWN:86} for exhaustive  treatments of exponential
families, and to Csisz\'{a}r and Mat\'{u}\v{s} \cite{CS:05,
CS:08}, Geyer \cite{GEYER:09}, and Rinaldo et al. \cite{R}  for
specialized results on the closure of exponential families
directly relevant to our problem.

Recall that we are interested in the exponential family of
probability distributions on $\mathcal{G}_n$ such that, for a
given choice of the natural parameters $\beta \in \mathbb{R}^2$,
the probability of observing a network $G_n \in \mathcal{G}_n$ is
\begin{equation}
\label{add}
    \PR_{n,\beta}(G_n)=\exp \left( n^2 \left(\langle \beta, t(G_n) \rangle -
    \psi_n(\beta)
    \right) \right), \quad \beta \in \mathbb{R}^2,
\end{equation}
    where $\psi_n(\beta)$ is the normalizing constant and the
    function $t(\cdot)$ is given in \eqref{eq:tGn}.
We remark that the above model assigns the same probability to all
graphs in $\mathcal{G}_n$ that have the same image under
$t(\cdot)$. We let $S_n = \{ t(G_n), G_n \in \mathcal{G}_n \}
\subset [0,1]^2$ be the set of all possible vectors of densities
of graph homomorphisms of $K_2$ and $K_3$ over the set
$\mathcal{G}_n$ of all simple graphs on $n$ nodes (see
(\ref{eq:tGn})).
    By (\ref{add}), the family on $\mathcal{G}_n$
    will induce the exponential family of probability
distributions $\mathcal{E}_n = \{
    \PR_{n, \beta}, \beta \in \mathbb{R}^2\}$
on $S_n$, such that the probability of observing a point $x \in
S_n$ is
\begin{equation}
    \PR_{n,\beta}(x) = \exp \left(n^2 \left(\langle \beta, x \rangle -
    \psi_n(\beta)
    \right)\right) \nu_n(x), \quad \beta \in \mathbb{R}^2,
\end{equation}
where $\nu_n(x) = |\{ t^{-1}(x) \}|$ is the measure on $S_n$
induced by the counting measure on $\mathcal{G}_n$ and $t(\cdot)$.
For each $n$, the family $\mathcal{E}_n$ has finite support not
contained in
    any lower dimensional set (see Lemma \ref{lem:Pn} below) and, therefore, is
    full and regular and, in particular,
    steep. 

    We will study the limiting behavior of sequences
    of models of the form
    $\{\PR_{n, \beta+ r o}\}$, where $\beta$ and $o$ are fixed vectors  in $\mathbb{R}^2$ and
    $n$ and $r$ are parameters both tending to
    infinity.
    While it may be tempting to regard $n$ as a
    surrogate for an increasing sample size, this would in fact be incorrect. Models parametrized by different
    values of $n$ and the same $r$ cannot be embedded (for the edge-triangle
    model) in any sequence of
    consistent probability measures, since the probability distribution
    corresponding to the smaller network  cannot in general be obtained from the other by
    marginalization, for a fixed choice of $\beta$. See \cite{SR:13} for details.
    We will show that different
choices of the direction $o$ will yield different extremal
behaviors of the model and we will categorize the variety of these
behaviors as a function of $o$ and, whenever it matters, of
$\beta$. A key feature of our analysis is the direct link to the
geometric properties of the polyhedral complex $\{C_k,
k=-1,0,\ldots\}$ defined by the set $P$ (see Section
\ref{extreme}).

Overall, the results of this section are obtained with non-trivial
extensions of techniques described in the exponential families
literature. Indeed, for fixed $n$, determining  the limiting
behaviors of the family $\mathcal{E}_n$ along sequences of natural
parameters $\{\beta + r o\}_{r \rightarrow \infty}$ for each unit
norm vector $o$ and each $\beta$ is the main technical ingredient
in computing the total variation closure of $\mathcal{E}_n$.
In particular Geyer \cite{GEYER:09} refers to the directions $o$
as the ``directions of recession" of the model. The relevance of
the directions of recession to the asymptotic behavior of
exponential random graphs is now well known, as demonstrated in
the work of Handcock \cite{H:03} and Rinaldo et al. \cite{R}.

\subsection{Finite $n$ geometry}
As we saw in Section \ref{mei}, the critical directions are
determined by the limiting object $P$. For finite $n$, an
analogous role is played by the convex support of $\mathcal{E}_n$,
which is given by the polytope
\[
    P_n = \mathrm{convhull}(S_n)\subset [0,1]^2.
\]
    The interior of $P_n$ is equal to all possible expected values of the sufficient
    statistics $\{\mathbb{E}_{n,\beta} \left(t(G_n)\right), \beta \in
    \mathbb{R}^2 \}$, where $\mathbb{E}_{n,\beta}$ is the expectation
    operator with respect to the measure $\mathbb{P}_{n,\beta}$.   Thus,
    it  provides a different parametrization
    of $\mathcal{E}_n$,
    known as the mean-value parametrization (see, e.g., \cite{BRN:78,BROWN:86}).
    Unlike the natural  parametrization, the mean value parametrization has
    explicit geometric properties that turn out to be particularly convenient in
    order to describe the closure of $\mathcal{E}_n$, and, ultimately, the
    asymptotics of the model.

    The next lemma characterizes the geometric properties of $P_n$. The most
    significant of these properties is that $\lim_n P_n = P$, an easy result that
    turns out to be the key for our analysis.
 Recall that we denote by $T(n,r)$ any Tur\'{a}n graph on $n$
nodes with $r$ classes. For $k=0,1,\ldots,n-1$, set $v_{k,n} =
t(T(n, k+1))$ and let $L_{k,n}$ denote the line segment joining
$v_{k,n}$ and $v_{k+1,n}$.

\begin{lemma}\label{lem:Pn}
    \begin{enumerate}
    \item The polytope $P_n$ is spanned by the points $\{ v_{k,n}, k=0,1,\ldots,
    \lceil n/2 \rceil -1\}$ and $v_{n-1,n}$. 
\item $\lim_n P_n = P$. \item If $n$ is a multiple of $(k +
1)(k+2)$, then $v_{k,n} = v_k$ and $v_{k+1,n} = v_{k+1}$. In
addition, for all such $n$, $L_{k,n} \cap S_n = \{v_k , v_{k+1}
\}$.
    \end{enumerate}
\end{lemma}

    Part 2. of Lemma \ref{lem:Pn} implies that for each $k$, $\lim_n v_{k,n} = v_k$,
    a fact that will be used in Theorem \ref{thm:main.2} to describe the
    asymptotics of the model along generic (i.e., non-critical directions). This conclusion still holds if the
polytopes $P_n$ are the convex hulls of isomorphism, not
homomorphism, densities. In this case, however, we have that $P_n
\supset P$ for each $n$ (see \cite{EN:11}). The seemingly
inconsequential fact stated in part 3. is instead of technical
significance for our analysis of the phase transitions along
critical directions, as will be described in Theorem
\ref{thm:main}. We take note that when $n$ is not a multiple of
$(k+1)(k+2)$, part 3. does not hold in general.

    \subsection{Asymptotics along generic directions}

    Our first result, which gives similar finding as in Section \ref{jump} shows that, for
     large $n$, if the distribution is parametrized by a vector
    with very large norm, then almost all of its mass will concentrate on the
    isomorphic class of a
    Tur\'{a}n graph (possibly the empty or the complete graph).  Which Tur\'{a}n graph it concentrates on
    will essentially depend on the ``direction"
    of the parameter vector with respect to the origin. Furthermore, there is an array of
    extremal directions that will give the same isomorphic class of a Tur\'{a}n
    graph.

\begin{theorem}\label{thm:main.2}
    Let $o$ and $\beta$ be vectors in $\mathbb{R}^2$ such that $o\neq o_k$ for
    $k=-1,0,1,\ldots$ and let $k$ be such  that $o \in C_k^{\circ}$. For any $0<\epsilon<1$
    arbitrarily small, there
exists an $n_0=n_0(\beta,\epsilon,o)>0$   such that the following
holds: for every $n>n_0$, there exists an
$r_0=r_0(\beta,\epsilon,o,n)>0$ such that for all $r>r_0$,
\[
    \PR_{n,\beta +r o}(v_{k,n})>1-\epsilon.
    \]
\end{theorem}

\begin{remark}
If in the theorem above we consider only values of $n$ that are
multiples of $(k+1)(k+2)$ then, by Lemma \ref{lem:Pn}, $v_{k,n} =
v_k$ for all such $n$, which implies convergence in total
variation to the point mass at $v_k$.
\end{remark}

 The theorem  shows that {\it any} choice of $o \in
C_k^{\circ}$ will yield the same asymptotic (in $n$ and $r$)
behavior, captured by the Tur\'{a}n graphon with $k+1$ classes. This
can be further strengthened to show that the convergence is
uniform in $o$ over compact subsets of $C_k^{\circ}$. See \cite{R}
for details. Interestingly, the initial value of $\beta$ does not
play any role in determining the asymptotics of $\PR_{n, \beta + r
o}$, which instead depends solely on which cone $C_k$ contains in
its interior  the direction $o$. Altogether, Theorem
\ref{thm:main.2} can be interpreted as follows: the interiors of
the cones of the infinite polyhedral complex $\{ C_k,
    k=-1,0,1,\ldots\}$ represent equivalence classes of ``extremal directions" of
    the model, whereby directions in the same class will parametrize, for
    large $n$ and $r$, the same degenerate distribution on some Tur\'{a}n graph.

%

    \subsection{Asymptotics along critical directions}
Theorem  \ref{thm:main.2} provides a complete categorization of
the asymptotics (in $n$ and $r$) of probability distributions of
the form $P_{n, \beta + r o}$ for any generic direction $o$ other
than the critical directions $\{ o_k, k=-1,0,1,\ldots\}$. We now consider the more delicate cases in which $o = o_k$ for some
$k$. Recall that, according to Theorem \ref{main}, in these instances the
typical graph drawn from the model will converge (as $n$ and $r$
grow and in the cut metric) to a large Tur\'{a}n graph, whose
number of classes is not entirely specified.

Our first result characterizes such behavior along subsequences
of the form $n = j (k+1)(k+2)$, for $j=1,2,\ldots$ and $k$ a
positive integer. Interestingly, and in contrast with Theorem
\ref{thm:main.2}, the limiting behavior along any critical
direction $o_k$ depends on $\beta$ in a discontinuous manner.
Before stating the result we will need to introduce some
additional notation. Let $l_{k} \in \mathbb{R}^2$ be the unit norm
vector spanning the one-dimensional linear subspace
$\mathcal{L}_{k}$ given by the line through the origin parallel to
$L_{k}$, where $k>0$, so $l_k$ is just a rescaling of the vector
\[
\left(1,\frac{k(3k+5)}{(k+1)(k+2)} \right).
    \]
 Next let  $H_k = \{ x
    \in \mathbb{R}^2 \colon \langle x, l_k \rangle = 0\} =
    \mathcal{L}_k^\bot$ be the line through the origin defining the
    linear subspace orthogonal to $\mathcal{L}_k$ and let
\begin{equation}
        H^+_k = \{ x \in \mathbb{R}^2 \colon \langle x, l_k \rangle >0 \}
        \quad \text{and} \quad      H^-_k = \{ x \in \mathbb{R}^2 \colon
    \langle x, l_k \rangle <0 \}
\end{equation}
be the positive and negative half-spaces cut out by $H_k$,
respectively.
    Notice that the linear subspace $\mathcal{L}^\bot_{k}$ is spanned by
    the vector $ o_k$ defined in \eqref{eq:outer}.

We will make the simplifying assumption that $n$ is a multiple of
$(k+1)(k+2)$. This implies, in particular, that $v_{k,n}= v_k$ and
$v_{k+1,n} = v_{k+1}$ are both vertices of $P_n$ and that the line
segment $L_{k,n} = L_k$ is a facet of $P_n$ whose normal cone is
spanned by the point $o_k$.

    \begin{theorem}\label{thm:main}
Let $k$ be a positive integer, $\beta \in \mathbb{R}^2$ and
$0<\epsilon<1$ be arbitrarily small. Then there exists an
$n_0=n_0(\beta,\epsilon,k)>0$ such that the following holds: for
every $n>n_0$ and a multiple of $(k+1)(k+2)$ there exists an $r_0=
r_0(\beta,\epsilon,k,n)>0$ such that for all $r>r_0$,
\begin{itemize}
    \item  if $\beta \in H^+_k$ or $\beta \in H_k$ then $\PR_{n,\beta+r
        o_k}(v_{k+1})>1-\epsilon$,
\item if $\beta \in H^-_k$ then $\PR_{n,\beta+r
o_k}(v_{k})>1-\epsilon$.

\end{itemize}
\end{theorem}

The previous result shows that, for large values of $r$ and $n$
(assumed to be a multiple of
    $(k+1)(k+2)$), the probability distribution $P_{n,\beta + r o_k}$
    will be
    concentrated almost entirely on either $v_k$ or $v_{k+1}$, depending on
    which side of $H_k$ the vector $\beta$ lies. In particular, the actual value
    of $\beta$ does not play any role in the asymptotics: only its position
    relative to $H_k$ matters. An interesting consequence of our result is
    the  discontinuity of the natural parametrization along the line
    $H_k$ in the limit as both $n$ and $r$ tend to infinity. This is in stark contrast to
    the limiting behavior of the same model when $n$ is infinity and $r$ tends to infinity: in
    this case the natural parametrization is a smooth (though non-minimal)
    parametrization.

    We now consider the critical directions $o_{-1}$
    and $o_{0}$ (see (\ref{eq:outer})), which are not covered by Theorem
    \ref{thm:main}.
    We will first describe the behavior of
    $\mathcal{E}_n$ along the direction of recession $o_{-1,n} :=\left(
    -1, \frac{n}{n-2}\right)$. This is the outer normal to the
    segment joining the vertices $v_{0,n} = (0,0)$ and $v_{n-1,n} = \left( 1 - \frac{1}{n}, \left(
    1 - \frac{1}{n}\right) \left( 1 - \frac{2}{n}\right) \right)$ of $P_n$,
    representing the empty and the complete graph, respectively. Notice that
    $o_{-1,n}
    \rightarrow o_{-1}$ as $n \rightarrow \infty$.
    In this case, we show that, for $n$ and $r$ large, the probability $P_{n, \beta+ r
        o_n}$, with $\beta = (\beta_1,\beta_2)$, assigns almost all of
        its mass to the empty graph when $\beta_1 + \beta_2 < 0$ and to the
        complete graph when $\beta_1 + \beta_2 > 0$, and it is uniform over
    $v_{0,n}$ and $v_{n-1,n}$ when $\beta_1 n(n-1)  +  \beta_2 (n-1)(n-2) =
    0$.


    \begin{theorem}
        \label{thm:main.3}
    Let $\beta=(\beta_1,\beta_2) \in \mathbb{R}^2$ be a fixed vector and $0<\epsilon <1$ be
    arbitrarily small. Then for every $n$ there exists an $r_0=
    r_0(\beta,\epsilon,n)$ such that, for all $r>r_0$, the total variation
    distance between $P_{n, \beta+r o_n}$ and the probability
    distribution which assigns to the points $v_{0,n}$ and $v_{n-1,n}$
    the probabilities
    \[
\frac{1}{1 + \exp\big( \beta_1 n(n-1)  +  \beta_2 (n-1)(n-2)
\big)}\
    \]
    and
    \[
\frac{ \exp\big( \beta_1 n(n-1)  +  \beta_2 (n-1)(n-2)  \big) }{ 1
+
    \exp\big( \beta_1 n(n-1)  +  \beta_2 (n-1)(n-2)  \big)
    },
        \]
        respectively, is less than $\epsilon$.
    \end{theorem}

In the last result of this section we will turn to the critical
direction $o_0=\left(0, -1 \right)$, which, for every $n \geq 2$,
is the outer normal to the horizontal facet of $P_n$ joining the
points $(0,0)$ and
\[
    v_{1,n} = \left(\frac{2 \lceil n/2 \rceil ( n
- \lceil n/2 \rceil)}{n^2}, 0\right),
\]
which we denote with $L_{0,n}$. Let $\mathcal{G}_{n,0}$ denote the
subset of $\mathcal{G}_n$ consisting of triangle free graphs. For
each $n$, consider the exponential family $\{\QR_{n,\beta_1} ,
\beta_1 \in \mathbb{R} \}$ of probability distributions on
$L_{0,n} \cap S_n$ given by
\begin{equation}
    \QR_{n,\beta_1}(x) =  \exp\left(n^2\left(\beta_1 x_1 -
    \phi_n(\beta_1)\right) \right)\nu_n(x), \quad x \in L_{n,0}\cap S_n, \,\, \beta_1 \in
    \mathbb{R},
\end{equation}
    where $\phi_n(\beta_1)$ is the normalizing constant and $\nu_n(x) = |\{ t^{-1}(x) \}|$ is the measure on $L_{0,n}$
induced by the counting measure on $\mathcal{G}_{n,0}$ and
$t(\cdot)$.

    \begin{theorem}\label{thm:main.4}
        Let $\beta=(\beta_1,\beta_2) \in \mathbb{R}^2$ be a fixed vector and
        $0<\epsilon<1$ an arbitrary number. Then for every $n$ there
        exists an $r_0=r_0(\beta, \epsilon, n)$ such that for all $r > r_0$ the
    total variation distance between $\PR_{n,\beta +r o_0}$ and
    $\QR_{n,\beta_1}$ is less than $\epsilon$.

    \end{theorem}

    When compared to Theorem \ref{ChDi}, Theorem \ref{thm:main.4} is less
    informative, as the class of triangle free graphs is larger than the class
    of subgraphs of the Tur\'{a}n graphs $T(n,2)$. We conjecture that this
    gap can indeed be resolved by showing that, for each $\beta_1$,
    $\QR_{n,\beta_1}$ assigns a vanishingly small mass to the set of all triangle
    free graphs that are not subgraphs of some $T(n,2)$ as $n \rightarrow
    \infty$. See \cite{EKR:76} for relevant results.

%

\subsection{From convergence in total variation to stochastic convergence in cut
distance}\label{sec:var.to.cut}
The results presented so far in Section \ref{ale} 
concern convergence in total variation of the homomorphism
densities of edges and triangles to point mass distributions at
points $v_{k,n}$.  They describe a rather different type of
asymptotic guarantees from the one obtained in Section \ref{mei},
whereby convergence occurs in the functional space
$\widetilde{\mathcal{W}}$ under the cut metric.  Nonetheless, both
sets of results are qualitatively similar and lend themselves to
nearly identical interpretations. Here we sketch how the total
variation convergence results imply convergence in probability in
the cut metric along subsequences. Notice that this is precisely
the same type of conclusions we obtained at the end of the
variational analysis, as remarked in Section \ref{sec:cut.to.cut}.

    We will let $n$ be of the form $j(k+1)(k+2)$ for $j=1,2,\ldots$. For
    simplicity we consider a direction $o$ in the interior of
    $C_k$, for some $k$. By Theorem \ref{thm:main.2} and using a standard diagonal argument,
    there exists a subsequence $\{ (n_i,r_i)\}_{i=1,2,\ldots}$ such that the sequence of
    probability measures $\{\mathbb{P}_{n_i,\beta+ r_io} \}_{i = 1,2,\ldots}$ converges in total
    variation to the point mass at $v_k$. Thus, for each $\epsilon>0$, there
    exists an $i_0 = i_0(\epsilon)$ such that, for all $i > i_0$, the probability
    that a random graph $G_{n_i}$ drawn from the probability distribution
    $\mathbb{P}_{n_i,\beta+ r_io}$ is such that $t(G_{i}) \neq v_k$ is less than
    $\epsilon$.
Let $\mathcal{A}_i$ be the event that $t(G_{i}) =
v_k$ and for notational convenience denote $\mathbb{P}_{n_i,\beta+
r_io}$ by $\mathbb{P}_i$. Thus, for each $i > i_0$, $\mathbb{P}_i(\mathcal{A}_i) > 1-\epsilon$.
    Let $H$ be any finite graph. Then, denoting with $\mathbb{E}_i$ the
    expectation with respect to $\mathbb{P}_i$ and with $1_{\mathcal{A}_i}$ the indicator function of
$\mathcal{A}_i$, we have
\[
        \mathbb{E}_i[t(H,G_i)]  =  \mathbb{E}_i[t(H,G_i) 1_{\mathcal{A}_i}]
        + \mathbb{E}_i[t(H,G_i) 1_{\mathcal{A}^c_i}],
        \]
    where
    \[
\mathbb{E}_i[t(H,G_i) 1_{\mathcal{A}_i}] = t(H,T(n_i,k+1)) =
t(H,f^{K_{k+1}}),
        \]
        since, given our assumption on the
    $n_i$'s, the point in $\widetilde{\mathcal{W}}$ corresponding to $T(n_i,k+1)$ is
$\tilde{f}^{K_{k+1}}$ for all $i$. Thus, using the  fact that
density homomorphisms are bounded by $1$,
\[
\mathbb{E}_i[t(H,G_i)] - t(H,f^{K_{k+1}}) =  \mathbb{E}_i[t(H,G_i)
    1_{\mathcal{A}^c_i}] \leq \mathbb{P}_i(\mathcal{A}^c_i) =  \epsilon
    \]
for all $i > i_0$.
Thus, we conclude
that $\lim_i \mathbb{E}_i  [t(H,G_i)] = t(H,f^{K_{k+1}})$ for each
finite graph $H$. By Corollary 3.2 in \cite{DJ:08}, as $i
\rightarrow \infty$,
\[
 \delta_{\square}(\tilde{f}^{G_i},
 \tilde{f}^{K_{k+1}}) \rightarrow 0 \quad \text{in probability.}
    \]

    Similar arguments apply to the case in which $o = o_k$ for some $k > 0$.
    Using instead Theorem \ref{thm:main}, we obtain that, if  $\{ G_{n_i}
\}_{i=1,2\ldots}$ is a sequence of random graphs drawn from the
sequence of probability distributions $\{ \mathbb{P}_{n_i,\beta+
r_io} \}$, then, as $i \rightarrow \infty$,
\[
 \delta_{\square}(\tilde{f}^{G_i},
 \tilde{f}^{K_{k+2}}) \rightarrow 0 \quad \text{if } \beta \in H^+_k \text{ or } \beta \in H_k,
    \]
    and
\[
 \delta_{\square}(\tilde{f}^{G_i},
 \tilde{f}^{K_{k+1}}) \rightarrow 0 \quad \text{if } \beta \in H^-_k
    \]
    in probability.

\section{Illustrative figures}
\label{illu}

\begin{figure}
\centering
\includegraphics[width=3in]{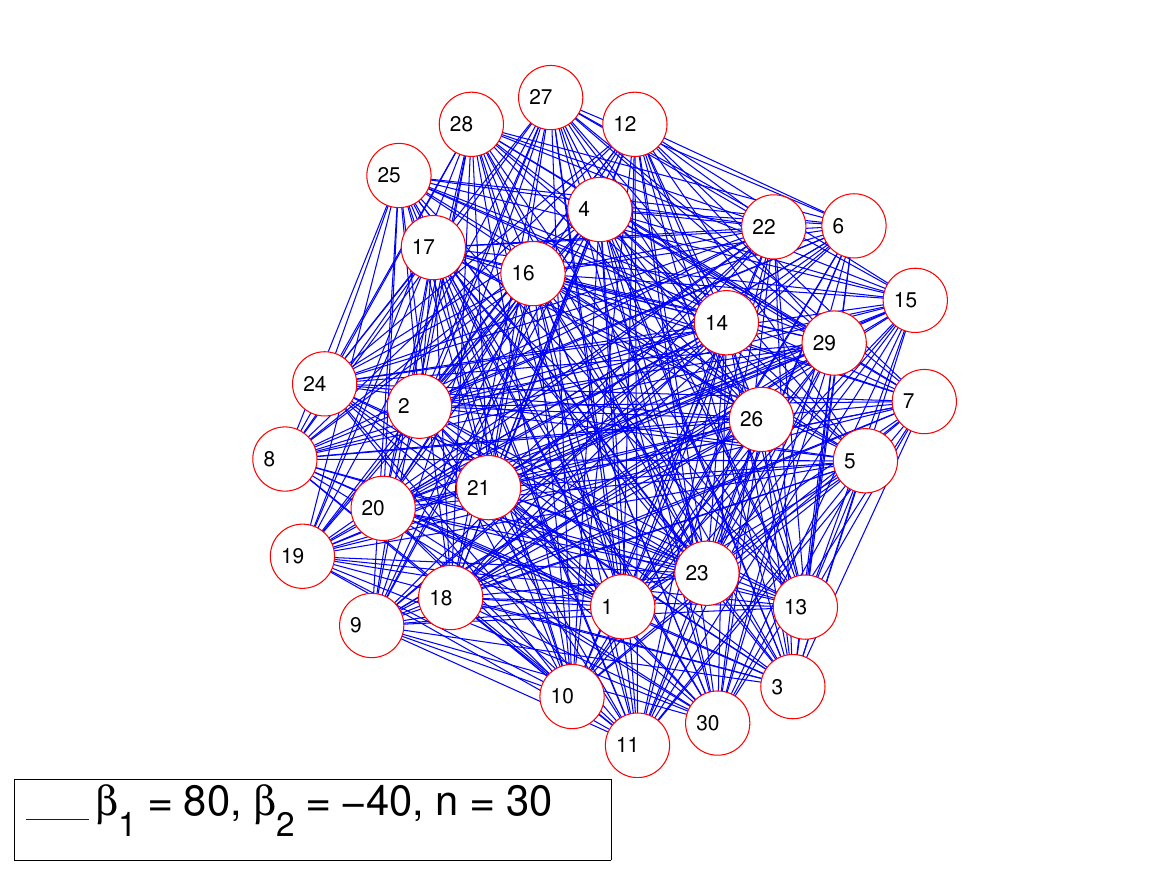}
\caption{A simulated realization of the exponential random graph
model on $30$ nodes with edges and triangles as sufficient
statistics, where the initial value $\beta=(0, 0)$, $r=80$, and
the generic direction $o=(1, -1/2)$ in $C_3^o$. The structure of
the simulated graph matches the predictions of Theorem
\ref{thm:main.2}.} \label{4}
\end{figure}

\begin{figure}
\centering
\includegraphics[width=3in]{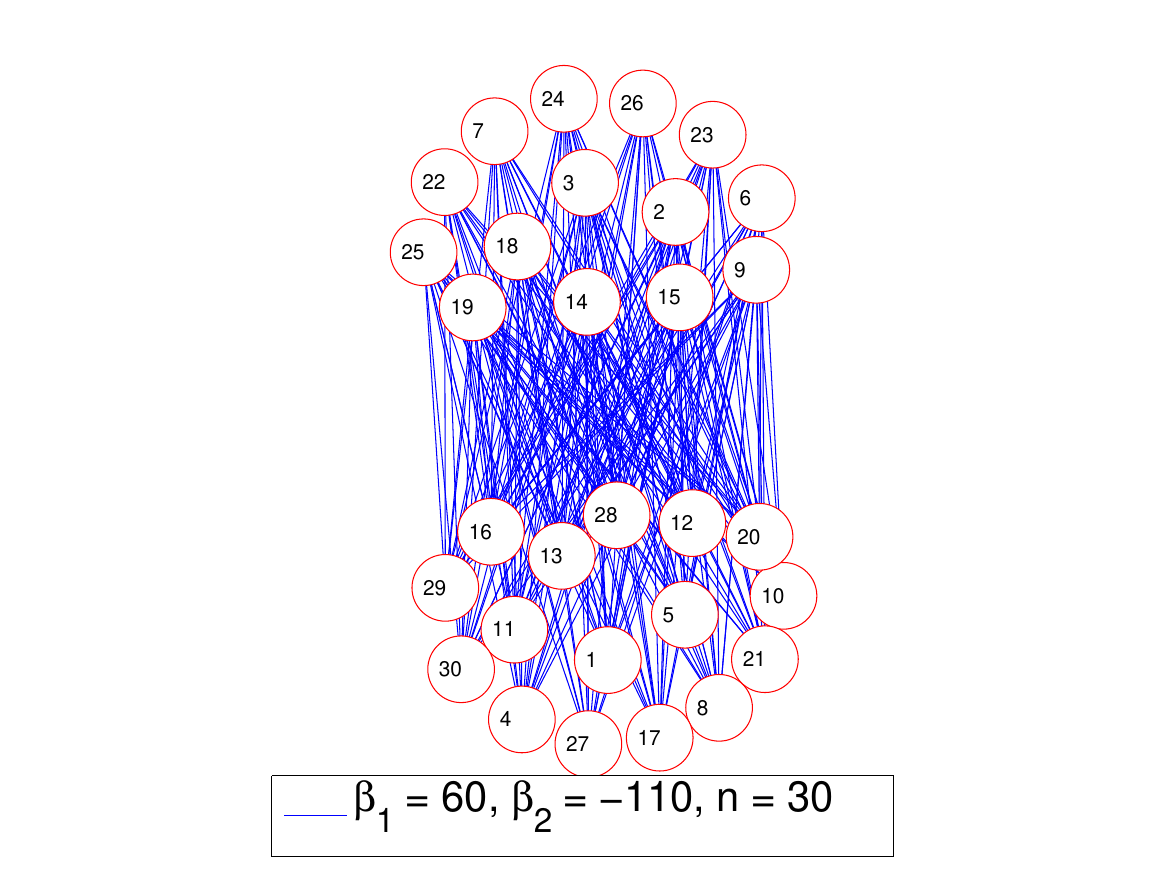}
\caption{A simulated realization of the exponential random graph
model on $30$ nodes with edges and triangles as sufficient
statistics, where the initial value $\beta=(20, -80)$ in $H_1^-$,
$r=40$, and the critical direction $o_1=(1, -3/4)$. The structure
of the simulated graph matches the predictions of Theorem
\ref{thm:main}.} \label{2}
\end{figure}

\begin{figure}
\centering
\includegraphics[width=3in]{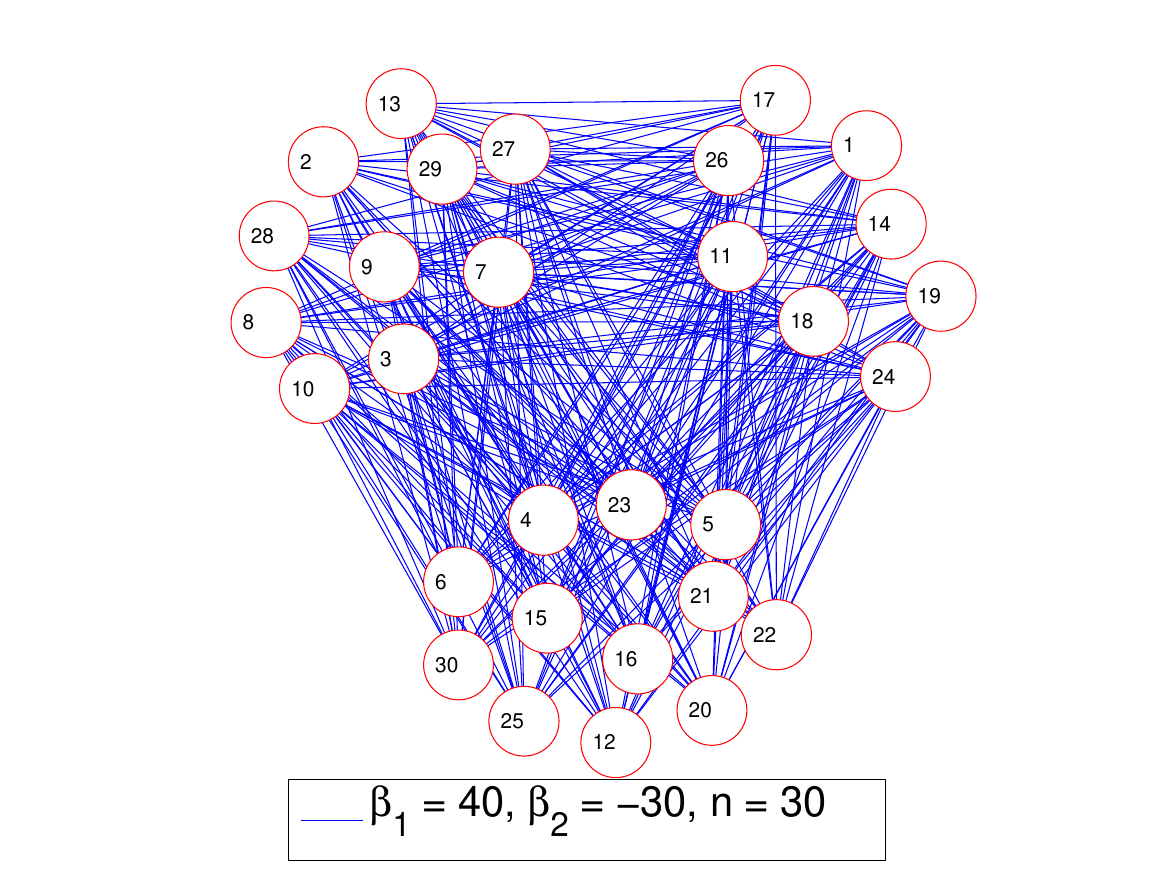}
\caption{A simulated realization of the exponential random graph
model on $30$ nodes with edges and triangles as sufficient
statistics, where the initial value $\beta=(0, 0)$ in $H_1$,
$r=40$, and the critical direction $o_1=(1, -3/4)$. The structure
of the simulated graph matches the predictions of Theorem
\ref{thm:main}.}  \label{3_1}
\end{figure}

\begin{figure}
\centering
\includegraphics[width=3in]{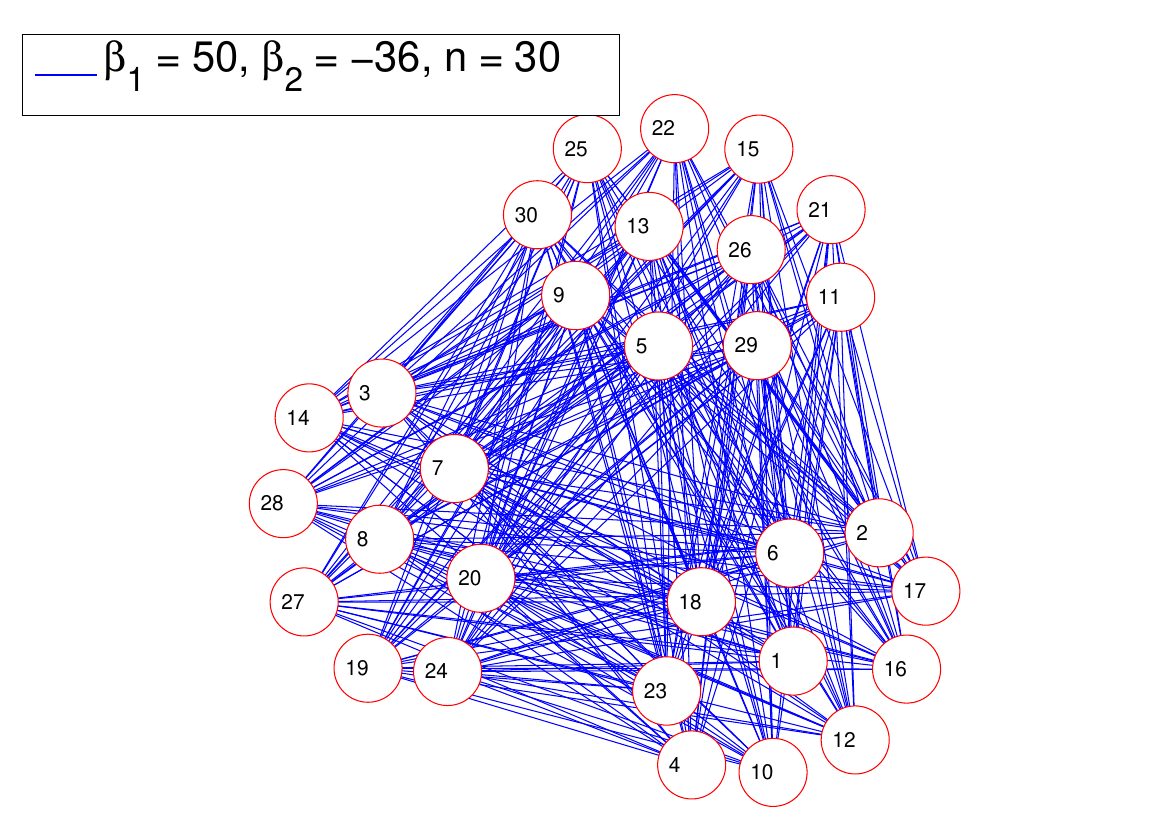}
\caption{A simulated realization of the exponential random graph
model on $30$ nodes with edges and triangles as sufficient
statistics, where the initial value $\beta=(10, -6)$ in $H_1^+$,
$r=40$, and the critical direction $o_1=(1, -3/4)$. The structure
of the simulated graph matches the predictions of Theorem
\ref{thm:main}.} \label{3_2}
\end{figure}

We have validated our theoretical findings with simulations of the
edge-triangle model under various specifications on the model
parameters. Figure \ref{4} depicts a typical realization from the
model when $n=30$ and $o$ is in $C_3^\circ$. As
predicted by Theorem \ref{thm:main.2}, the resulting graph is
complete equipartite with $4$ classes. Figures \ref{2}, \ref{3_1}
and \ref{3_2} exemplify the results of Theorem
\ref{thm:main}. For these simulations, we consider the critical
direction $o_1 = (1, -3/4)$ and again a network size of $n=30$,
and then vary the initial values of $\beta$. Figures \ref{2} and
\ref{3_2} show respectively the outcome of two typical draws when
$\beta$ is in $H_1^-$ and $H_1^+$, respectively. As predicted by
our theorem, we obtain a complete bipartite and tripartite graph.
Figure \ref{3_1} depicts instead the case of $\beta$ exactly along
the hyperplane $H_1$, for which, according to our theory, a typical
realization would again be a complete tripartite graph.

As a final remark, simulating from the extremal parameter
configurations we described using off-the-shelf MCMC methods (see,
e.g., \cite{geyer-thompson, ergm} and, for a convergence result,
\cite{CD}) is quite difficult. This is due to the fact
that under these extremal settings, the model places most of its
mass on only one or two types of Tur\'{a}n graphs, and the chance
of a chain being able to explore adequately the space of graphs
using local moves and to eventually reach the configuration of
highest energy is essentially minuscule.

\section{Further discussions}
\label{further}
As shown by Bhamidi et al. \cite{B} and Chatterjee and Diaconis \cite{CD}, as $n\rightarrow \infty$, when $\beta_2$ is positive, a typical graph drawn from
the standard edge-triangle exponential random graph model (\ref{eq:edge-triangle}) has a somewhat trivial structure: it always looks like an Erd\H{o}s-R\'{e}nyi random graph or a mixture of Erd\H{o}s-R\'{e}nyi random graphs. By raising the triangle density to an exponent $\gamma>0$, Lubetzky and Zhao \cite{LZ} proposed a natural generalization:
\begin{equation}
\label{gpmf}
\PR_{n,\beta}(G_n)=\exp\left(n^2(\beta_1
t(K_2,G_n)+\beta_2 t(K_3,G_n)^\gamma-\psi_n(\beta))\right),
\end{equation}
which enabled the
model to exhibit a non-trivial structure even when $\beta_2$ is positive.
This generalized model still features the Erd\H{o}s-R\'{e}nyi behavior if $\gamma\geq 2/3$; but for $\gamma<2/3$, there exist regions of values of $(\beta_1,\beta_2)$ for which a typical graph drawn from the model has symmetry breaking. We are interested to know how the \textit{double asymptotic} framework discussed in the earlier sections would lend itself to this generalized model.

Below we adapt our first main result for the standard model (Theorem \ref{thm.at}) to the generalized model and carry out some explicit calculations. As we will see, it conforms to the findings in \cite{LZ} and gives the limiting graphon structure for the solution of the variational problem. The proof of the theorem offers one explanation for why $2/3$ is a separating value for the exponent $\gamma$: it is intimately tied to the upper boundary of the feasible edge-triangle homomorphism densities. Furthermore, the theorem provides convincing evidence that the region of symmetry breaking for the generalized edge-triangle model is potentially much larger than the ones depicted on page 5 of \cite{LZ}. We remark that, using similar arguments, it is also possible to adapt our other results for the standard model to the generalized model, but the calculations would be rather involved, especially when they concern the lower boundary of the feasible edge-triangle homomorphism densities. Recall that for a non-negative constant $c$, we write $u=c$ when $u$ is the constant graphon with value $c$.

\begin{theorem}
\label{main.further}
Consider the generalized edge-triangle exponential random graph model (\ref{gpmf}). Let $\beta_1=a\beta_2+b$. Then
\begin{equation}
    \lim_{\beta_2 \rightarrow \infty} \sup_{\tilde{f} \in \tilde{F}^*(\beta_2) }
    \delta_{\square} (\tilde{f}, \tilde{U})=0,
\end{equation}
where for $\gamma \geq 2/3$, the set $U \subset \mathcal{W}$ is determined as follows:
\begin{itemize}
    \item $U=\{1\}$ if $a>-1$ or $a=-1$ and $b>0$,
    \item $U=\{0, 1\}$ if
$a=-1$ and $b=0$, and \item $U=\{0\}$ if $a<-1$ or $a=-1$ and
$b<0$;
\end{itemize}
and for $\gamma<2/3$, the set $U \subset \mathcal{W}$ is determined as follows:
\begin{itemize}
    \item $U=\{1\}$ if $a\geq -\frac{3}{2}\gamma$, and
    \item $U=\{f\}$ if $a<-\frac{3}{2}\gamma$,
\end{itemize}
where $f(x,y)=\left\{%
\begin{array}{ll}
    1 & \hbox{if $0\leq x,y\leq \left(-\frac{2a}{3\gamma}\right)^{\frac{1}{3\gamma-2}}$,} \\
    0 & \hbox{otherwise.} \\
\end{array}%
\right.$
\end{theorem}

\section*{Acknowledgements}
The authors benefited from participation in the 2013 workshop on
Exponential Random Network Models at the American Institute of
Mathematics. They thank the anonymous referees and the AE for their useful
comments and suggestions and Ed Scheinerman of the Johns Hopkins
University for making available the Matgraph toolbox that was used
as a basis for many of their MATLAB simulations. Mei Yin
gratefully acknowledges useful discussions with Charles Radin and
Lorenzo Sadun.

\section{Proofs}

\begin{proof}[Proof of Theorem \ref{thm.at}]
Suppose $H_2$ has $p$ edges. Subject to $\beta_1=a\beta_2+b$, the
variational problem (\ref{lmax}) in the Erd\H{o}s-R\'{e}nyi region
takes the following form: Find $u$ so that
\begin{equation}
\beta_2(au+u^p)+bu-I(u)
\end{equation}
is maximized. Take an arbitrary sequence $\beta_2^{(i)}
\rightarrow \infty$. Let $u_i$ be a maximizer corresponding to
$\beta_2^{(i)}$ and $u^*$ be a limit point of the sequence $\{ u_i
\}$. By the boundedness of $bu$ and $I(u)$, we see that $u^*$ must
maximize $au+u^p$. For $a\neq -1$, this maximizer is unique, but
for $a=-1$, both $0$ and $1$ are maximizers. In this case,
$\beta_2(au+u^p)=0$, so we check the value of $bu-I(u)$ as well.
We conclude that $u^*=1$ for $b>0$, $u^*=0$ for $b<0$, and $u^*$
may be either $1$ or $0$ for $b=0$.
\end{proof}

The following lemma appeared as an exercise in \cite{L2}.

\begin{lemma}[Lov\'{a}sz]
\label{chromatic} Let $F$ and $G$ be two simple graphs. Let $f$ be
a graphon such that $t(F, G)>0$ and $t(G, f)>0$. Then $t(F, f)>0$.
\end{lemma}

\begin{proof}
Suppose $|V(F)|=m$ and $|V(G)|=n$. Since $t(G, f)>0$,
there is a Lebesgue measurable set $A \subseteq \mathbb{R}^n$ and
$|A| \neq 0$ such that $\prod_{\{i,j\}\in E(G)}f(x_i, x_j)\\>0$ for
$x \in A$. Since $t(F, G)>0$, there exists a graph homomorphism
$h: V(F) \rightarrow V(G)$. Since $F$ and $G$ are both labeled graphs, this naturally induces a map $h': \mathbb{R}^m \rightarrow \mathbb{R}^n$. If $h$ is one-to-one, take $B=(h')^{-1}(A) \subseteq \mathbb{R}^m$. Then clearly $|B| \neq
0$ is Lebesgue measurable and $\prod_{\{i,j\}\in E(F)}f(y_i, y_j)>0$ for $y \in B$. If $h$ is not one-to-one, identifying $B \subseteq \mathbb{R}^m$ with $|B| \neq 0$ requires treating the vertices of $V(F)$ that map to the same vertex of $V(G)$ under $h$ as independent coordinates. We illustrate this procedure through a simple example. Suppose $F$ is a two-star consisting of edges $\{1,2\}$ and $\{1,3\}$ and $G$ is a single edge $\{1,2\}$. A graph homomorphism between the two vertex sets $V(F)$ and $V(G)$ may be given by $1\mapsto 1, 2\mapsto 2, 3\mapsto 2$. Say we have found a Lebesgue measurable set $A=\{(x_1,x_2): a\leq x_1\leq b, c(x_1)\leq x_2\leq d(x_1)\} \subseteq \mathbb{R}^2$ such that $f(x_1, x_2)>0$ for $x \in A$. Take $B=\{(y_1, y_2, y_3):a \leq y_1 \leq b, c(y_1)\leq y_2\leq d(y_1), c(y_1)\leq y_3\leq d(y_1)\} \subseteq \mathbb{R}^3$. Then clearly $|B| \neq 0$ is Lebesgue measurable and $f(y_1, y_2)f(y_1, y_3)>0$ for $y \in B$. It follows that $t(F, f)>0$.
\end{proof}

\begin{proof}[Proof of Theorem \ref{ChDi}]
Take an arbitrary sequence $\beta_2^{(i)} \rightarrow -\infty$.
For each $\beta_2^{(i)}$, we examine the corresponding variational
problem (\ref{setmax}). Let $\tilde{f}_i$ be an element of
$\tilde{F}^*(\beta_2^{(i)})$. Let $\tilde{f}^*$ be a limit point
of $\tilde{f}_i$ in $\widetilde{\mathcal{W}}$ (its existence is
guaranteed by the compactness of $\widetilde{\mathcal{W}}$).
Suppose $t(H_2, f^*)>0$. Then by the continuity of $t(H_2, \cdot)$
and the boundedness of $t(H_1, \cdot)$ and
$\iint_{[0,1]^2}I(\cdot)dxdy$, $\lim_{i \rightarrow
\infty}\psi_\infty(\beta_2^{(i)})=-\infty$. But this is impossible
since $\psi_\infty(\beta_2^{(i)})$ is uniformly bounded below, as
can be easily seen by considering $f^{K_{r-1}}$ as a test function (see (\ref{eq:turan-graphon})), where $K_r$ denotes a complete graph on $r$
vertices. Thus $t(H_2, f^*)=0$. Since $H_2$ has chromatic number $r$, $t(H_2, K_r)>0$,
which implies that $t(K_r, f^*)=0$ by Lemma \ref{chromatic}. By the graphon version of
Tur\'{a}n's theorem for $K_r$-free graphs \cite{Pik}, the edge density $e$ of
$f^*$ must satisfy $e=t(H_1, f^*)\leq (r-2)/(r-1)$. This implies
that the measure of the set $\{(x, y) \in [0, 1]^2|f^*(x, y)>0\}$
is at most $(r-2)/(r-1)$. Otherwise, the
graphon $\bar{f}(x, y)=\left\{%
\begin{array}{ll}
    1 & \hbox{$f^*(x, y)>0$,} \\
    0 & \hbox{otherwise} \\
\end{array}%
\right.$ would be $K_r$-free but with edge density greater than
$(r-2)/(r-1)$, which is impossible.

Take an arbitrary edge density $e \leq (r-2)/(r-1)$. We consider
all graphons $f$ such that $t(H_1, f)=e$ and $t(H_2, f)=0$.
Subject to these constraints, maximizing (\ref{setmax}) is
equivalent to minimizing $\iint_{[0,1]^2}I(f)dxdy$. Since $t(H_2,
f)=0$, as argued above, the set $A=\{(x, y) \in [0, 1]^2|f(x,
y)>0\}$ has measure at most $(r-2)/(r-1)$. If the measure of $A$
is less than $(r-2)/(r-1)$, we randomly group part of the set $[0,
1]^2-A$ into $A$ so that the measure of $A$ is exactly
$(r-2)/(r-1)$. We note that
\begin{equation}
\iint_{[0, 1]^2} I(f(x, y))dxdy=\iint_{A}I(f(x, y))dxdy.
\end{equation}
More importantly, since $I(\cdot)$ is convex, by Jensen's
inequality, we have
\begin{equation}
\iint_{A}I(f(x, y))dxdy \geq \frac{r-2}{r-1} I\left(\iint_{A}
\frac{r-1}{r-2}f(x, y) dxdy\right)=\frac{r-2}{r-1}
I\left(\frac{r-1}{r-2}e\right),
\end{equation}
where the first equality is obtained only when $f(x, y)\equiv
e(r-1)/(r-2)$ on $A$.

The variational problem (\ref{setmax}) is now further reduced to
the following: Find $e\leq (r-2)/(r-1)$ (and hence $f(x, y)$) so
that
\begin{equation}
\beta_1 e-\frac{r-2}{r-1}I\left(\frac{r-1}{r-2}e\right)
\end{equation}
is maximized. Simple computation yields $e=p(r-2)/(r-1)$, where
$p=e^{2\beta_1}/(1+e^{2\beta_1})$. Thus $pf^{K_{r-1}}$ is a maximizer for (\ref{setmax}) as $\beta_2
\rightarrow -\infty$. We claim that any other maximizer $h$ (if it
exists) must lie in the same equivalence class. Recall that $h$
must be $K_r$-free. Also, $h$ is zero on a set of measure
$1/(r-1)$ and $p$ on a set of measure $(r-2)/(r-1)$. The graphon $\bar{h}(x, y)=\left\{%
\begin{array}{ll}
    1 & \hbox{$h(x, y)=p$,} \\
    0 & \hbox{otherwise} \\
\end{array}%
\right.$ describes a $K_r$-free graph with edge density
$(r-2)/(r-1)$. By the graphon version of Tur\'{a}n's theorem \cite{Pik}, $\bar{h}$ corresponds to
the complete $(r-1)$-equipartite graph, and is thus equivalent to
$f^{K_{r-1}}$. Hence $h=p\bar{h}$ is equivalent to $pf^{K_{r-1}}$.
\end{proof}

\begin{proof}[Proof of Theorem \ref{main}]
Subject to $\beta_1=a\beta_2+b$, the variational problem
(\ref{setmax}) takes the following form: Find $f(x, y)$ so that
\begin{equation}
\beta_2 (ae+t)+be-\iint_{[0,1]^2}I(f(x, y))dxdy
\end{equation}
is maximized, where $e=t(H_1, f)$ denotes the edge density and
$t=t(H_2, f)$ denotes the triangle density of $f$, respectively.
Take an arbitrary sequence $\beta_2^{(i)} \rightarrow -\infty$.
Let $\tilde{f}_i$ be an element of $\tilde{F}^*(\beta_2^{(i)})$.
Let $\tilde{f}^*$ be a limit point of $\tilde{f}_i$ in
$\widetilde{\mathcal{W}}$ (its existence is guaranteed by the
compactness of $\widetilde{\mathcal{W}}$). By the continuity of
$t(H_2, \cdot)$ and the boundedness of $t(H_1, \cdot)$ and
$\iint_{[0,1]^2}I(\cdot)dxdy$, we see that $f^*$ must minimize
$ae+t$. This implies that $f^*$ must lie on the Razborov curve
(i.e., lower boundary of the feasible region) (see Figure
\ref{et}). Note further that $ae+t$ is a linear function, so $f^*$ must minimize over the convex hull $P$ of $R$ (see Figure \ref{p}).
Since $R$ and $P$ only intersect at the points $v_k$, $k=1,2,\ldots$, $f^*$ corresponds to a Tur\'{a}n graphon with $k$ classes.

Consider two adjacent points $v_k=(e_k, t_k)$ and
$v_{k+1}=(e_{k+1}, t_{k+1})$, where
\begin{equation}
(e_k, t_k)=\left(\frac{k}{k+1}, \frac{k(k-1)}{(k+1)^2}\right)
\text{ and } (e_{k+1}, t_{k+1})=\left(\frac{k+1}{k+2},
\frac{k(k+1)}{(k+2)^2}\right).
\end{equation}
Let $L_k$ be the line segment joining these two points. The slope
of the line passing through $L_k$ is
\begin{equation}
\frac{k(3k+5)}{(k+1)(k+2)}=-a_k.
\end{equation}
It is clear that $a_k$ is a decreasing function of $k$ and $a_k
\rightarrow -3$ as $k \rightarrow \infty$. More importantly, if
$a>a_k$, then $ae_k+t_k<ae_{k+1}+t_{k+1}$; if $a=a_k$, then
$ae_k+t_k=ae_{k+1}+t_{k+1}$; and if $a<a_k$, then
$ae_k+t_k>ae_{k+1}+t_{k+1}$. Decreasing $a$ thus moves the
location of the minimizer $f^*$ upward,
with sudden jumps happening at special angles $a=a_k$, where the
sign of $b$ comes into play as in the proof of Theorem
\ref{thm.at}.
\end{proof}

\begin{proof}[Proof of Lemma \ref{lem:Pn}]
For part 1., the proof of Theorem 2  in \cite{B:76} implies that
any linear functional of the form $L_{\gamma}(x) = \langle x, c
\rangle$, where $c = (1, \gamma)^\top$ with $\gamma \in
\mathbb{R}$, is maximized over $P_n$ by some $v_{k,n}$ and,
conversely, any point $v_{k,n}$ is such that
\begin{equation}
    v_{k,n} = \mathrm{argmax}_{x \in P_n} L_{\gamma}(x)
\end{equation}
    for some $\gamma \in \mathbb{R}$. Thus, $P_n$ is the convex hull of the points $\{
    v_{k,n},k=0,1,\ldots,n-1\}$.
  Next, if $r \geq \lceil n/2 \rceil $, the size of the
  larger class(es) of any $T(n, r)$ is $2$ and
   the size of the smaller class(es) (if any) is $1$. Thus, the increase in the
   number of edges and triangles going from $T(n, r)$ to $T(n, r+1)$ is $1$ and
   $(n-2)$, respectively. As a result, the points $\{ t(T(n, r)), r = \lceil n/2
   \rceil, \ldots, n \}$ are collinear.

To show part 2., notice that, by definition, $P_n \subset P$, so
it is enough to show that for any $x \in P$ and $\epsilon > 0$
there exists an $n' = n'(x,\epsilon)$ such that $\inf_{y \in P_n}
\| x - y\| < \epsilon$ for all $n > n'$. But this follows from the
fact that, for each fixed $k$, $\lim_{n \rightarrow \infty}
v_{k,n} = v_k$ and every $x \in P$ is either an extreme point of
$P$ or is contained in the convex hull of a {\it finite} number of
extreme points of $P$.

The first claim of part 3. can be directly verified with easy
algebra (see (\ref{eq:tGn})). The second claim follows from
Theorem 4.1 in \cite{Razborov} and the strict concavity of the
lower boundary of $R$ on each subinterval $\left[ ( k-1)/k,k/(k+1)
\right]$.
\end{proof}

    The key steps of the proofs of Theorems \ref{thm:main.2} and, in
    particular, \ref{thm:main} rely on a careful
    analysis of the closure of the exponential family corresponding to the model
    under study. For the sake of clarity, we will provide a self-contained
    treatment. For details, see \cite{CS:05, CS:08, GEYER:09, R}.

    \subsubsection*{The closure of $\mathcal{E}_n$}
    Fix a positive integer $k$. We first describe the total variation closure of the family
    $\mathcal{E}_n$ for $n$ tending to infinity as $n =
j(k+1)(k+2)$, for $j=1,2,\ldots$. As a result, for all such $n$,
$v_{k,n} = v_k$ and $v_{k+1,n} = v_{k+1}$, which implies $L_{k,n}
= L_k$ (see Lemma \ref{lem:Pn}).


 Let $\nu_{k,n}$ be the restriction of   $\nu_n$ to $L_{k}$ and consider the exponential family on $P_n \cap
    L_{k,n} = \{ v_{k}, v_{k+1} \}$ generated by $\nu_{k,n}$ and $t$, and
    parametrized by $\mathbb{R}^2$, which we
    denote with $\mathcal{E}_{k,n}$.
    Thus, the probability of observing the point $x \in  \{ v_{k},
    v_{k+1} \}$ is
\begin{equation}\label{eq:exp2}
    \PR_{n,k,\beta}(x) = \frac{e^{ n^2 \langle x, \beta \rangle}}{e^{n^2 \langle v_{k},
\beta \rangle} \nu_n(v_{k})+ e^{n^2\langle v_{k+1}, \beta \rangle}
\nu_n(v_{k+1})} \nu_{n}(x), \quad \beta \in
    \mathbb{R}^2.
\end{equation}

The new family $\mathcal{E}_{k,n}$ is an element of the closure of
$\mathcal{E}_n$ in the topology corresponding to the variation
metric. More precisely, the family $\mathcal{E}_{k,n}$ is
comprised by all the limits in total variation of sequences of
distributions from $\mathcal{E}_n$ parametrized by sequences $\{
\beta^{(i)}\} \subset \mathbb{R}^2$ such that $\lim_i
\|\beta^{(i)}\| = \infty$ and $\lim_i
\frac{\beta^{(i)}}{\|\beta^{(i)} \|} = \frac{o_k}{\| o_k
\|}$.

\begin{proposition}\label{prop:closure}
    Let $n$ be fixed and a multiple of $(k+1)(k+2)$. For any $\beta \in
    \mathbb{R}^2$, consider the sequence of parameters
    $\{\beta^{(i)}  \}_{i=1,2,\ldots}$ given by
    $\beta^{(i)} = \beta + r_i o_{k}$, where $\{r_i\}_{i=1,2,\ldots}$ is a sequence of positive
numbers tending to infinity. Then,
\[
    \lim_i \PR_{n,\beta^{(i)}} (x) = \left\{
    \begin{array}{ll}
        \PR_{n,k,\beta}(x) & \text{if }  x \in \{
    v_{k},v_{k+1}\},\\
    0 & \text{if }x \in S_n \setminus \{v_k, v_{k+1}\}.\\
\end{array}
\right.
    \]
    In particular, $\PR_{n,\beta^{(i)}}$ converges in total variation to
    $\PR_{n,k,\beta}$ as $i \rightarrow \infty$.
\end{proposition}
\begin{proof}
The proof can be found in, e.g., \cite{CS:05, R}. We provide it
for completeness.

Let $x^* \in S_n$. Then, for any $\beta \in \mathbb{R}^2$,
\[
    \lim_{i \rightarrow \infty} \PR_{n,\beta^{(i)}}(x^*) = \frac{e^{n^2\langle
    x^*,\beta \rangle }  }{
    \lim_{i \rightarrow \infty}     e^{n^2\psi_n(\beta^{(i)}) - r_i n^2\langle
    x^*,o_k \rangle  }} \nu_n(x^*).
    \]
    First suppose that $x^* \in \{ v_k, v_{k+1}\}$. Then,

    \begin{multline}
    e^{n^2\psi_n(\beta^{(i)}) - r_i n^2 \langle x^*, o_k \rangle  }\\\notag=\sum_{x \in S_n
    \setminus \{ v_k,v_{k+1}\}}
    e^{n^2\langle x, \beta \rangle + r_i n^2\langle x - x^*, o_k  \rangle  }  \nu_n(x)
    +\sum_{x \in \{ v_k, v_{k+1}\}}  e^{n^2\langle x, \beta \rangle }
    \nu_n(x)\\\notag
     \downarrow  e^{n^2\langle v_{k},
\beta \rangle} \nu_n(v_{k})+ e^{n^2\langle v_{k+1}, \beta \rangle}
\nu_n(v_{k+1}),\notag
\end{multline}
as $i    \rightarrow \infty$, because $\sum_{x \in S_n \setminus
\{ v_k,v_{k+1}\}} e^{n^2\langle x, \beta \rangle + r_i n^2\langle
x - x^*, o_k  \rangle  }  \nu_n(x) \downarrow 0$.
    This follows easily since the term $\langle x -  x^*, o_k \rangle$ is $0$ if $x \in \{
    v_k, v_{k+1}\}$ and is strictly negative otherwise.
    Thus, $\PR_{n,\beta^{(i)}}(x^*)$ converges to $\PR_{n,k,\beta}(x^*)$
    (see (\ref{eq:exp2})).

    If $x^* \in S_n \setminus \{ v_k, v_{k+1}\}$, since $P_n$ is full
    dimensional, we have instead
    \[
    e^{n^2\psi_n(\beta^{(i)}) - r_i n^2\langle x^*, o_k \rangle  } \geq  \sum_{x \in
    S_n \colon \langle x - x^*, o_k \rangle > 0} e^{n^2\langle x, \beta \rangle +
    r_i n^2\langle x - x^*, o_k  \rangle  }  \nu_n(x) \rightarrow \infty,
        \]
    as $i \rightarrow \infty$. Therefore $\PR_{n,\beta^{(i)}}(x^*)
        \rightarrow 0$. The proof is now complete.
\end{proof}

The parametrization \eqref{eq:exp2} is thus redundant, as it
requires two parameters  to represent a distribution whose support
lies on a $1$-dimensional hyperplane. One parameter is all that is
needed to describe this distribution, a reduction that can be
accomplished by standard arguments. Because such reparametrization
is highly relevant to our problem, we provide the details.

\begin{proposition}
    \label{prop:minimal}
    The family $\mathcal{E}_{k,n}$ is a one-dimensional exponential family
    parametrized by $\mathcal{L}_{k}$. Equivalently, $\mathcal{E}_{k,n}$ can be
    parametrized with $\{\langle l_{k}, \beta \rangle,\\ \beta \in \mathbb{R}^2\}=\mathbb{R}$ as follows:
     \begin{equation}\label{eq:reduced}
    \PR_{n,k,\beta}(x) = \frac{e^{n^2\langle x, l_{k} \rangle \cdot \langle
    l_{k}, \beta \rangle }   }{ e^{n^2\langle v_{k}, l_{k} \rangle \cdot
    \langle l_{k}, \beta \rangle}\nu_n(v_{k})  +e^{ n^2\langle v_{k+1}, l_{k} \rangle
    \cdot \langle l_{k}, \beta \rangle} \nu_n(v_{k+1}) } \nu_{n}(x),
    \end{equation}
where $ x \in \{v_{k},v_{k+1} \}$.
\end{proposition}
\begin{proof}
For an $x \in \mathbb{R}^2$ and a linear subspace $\mathcal{S}$ of
$\mathbb{R}^2$, let $\Pi_{\mathcal{S}}(x)$ be the orthogonal
projection of $x$ onto $\mathcal{S}$ with respect to the Euclidean
metric. Set $\tilde{o}_{k} = \frac{o_k}{\|o_k\|}$ and let
$\alpha_k  \in \mathbb{R}$ define the one-dimensional hyperplane
(i.e., the line) going through $L_k$, i.e., $\{ x \in \mathbb{R}^2
\colon \langle x, \tilde{o}_{k}
    \rangle = \alpha_k \}$.
Then for every $\beta \in \mathbb{R}^2$ and $x \in  \{ v_{k},
    v_{k+1} \}$, we have
    \[
    \langle x, \beta \rangle = \langle \Pi_{\mathcal{L}_{k}} x, \beta
    \rangle + \langle \Pi_{\mathcal{L}_{k}^\bot} x , \beta \rangle =
    \langle x, l_{k} \rangle \cdot \langle l_{k}, \beta \rangle + \alpha_k
    \langle \tilde{o}_{k}, \beta \rangle
    \]
    since $\alpha_k = \langle v_{k}, \tilde{o}_{k} \rangle = \langle
    v_{k+1}, \tilde{o}_{k}
    \rangle$.     Plugging into \eqref{eq:exp2}, we obtain
    \eqref{eq:reduced}.
From that equation we see that, for any pair of distinct parameter
vectors $\beta$ and $\beta'$, $\PR_{n,k,\beta} = \PR_{n,k,\beta'}$
if and only if $\langle
    l_k, \beta \rangle = \langle l_k, \beta' \rangle$, i.e., if and
    only if they project to the same point in $\mathcal{L}_k$.  This proves the claim.
\end{proof}

\begin{remark}
The geometric interpretation of Proposition
    \ref{prop:minimal} is the
    following: $\beta$
    and $\beta'$ parametrize the same distribution on $\mathcal{E}_{k,n}$ if
    and only if the
    line going through them is parallel to the line spanned by $o_{k}$.
\end{remark}

    Finally, same arguments used in the proof of Proposition
    \ref{prop:closure} also imply that the closure of $\mathcal{E}_n$ along  generic
    (i.e., non-critical) directions  is comprised of point masses at the points
    $v_{k,n}$. For the next result, we do not need the condition of $n$
    being multiple of $(k+1)(k+2)$.
    \begin{corollary}\label{cor:o}
       Let $o \in \mathbb{R}^2$ be different from $o_j$, $j=-1,0,1,\ldots$
        and let $k$ be such that $o \in C_k^{\circ}$. There exists an $n_0 =
        n_0(o)$ such that, for any fixed $n>n_0$, 
and any sequence of parameters
    $\{\beta^{(i)}  \}_{i=1,2,\ldots}$ given by
    $\beta^{(i)} = \beta + r_i o$, where $\{r_i\}_{i=1,2,\ldots}$ is a sequence of positive
    numbers tending to infinity and $\beta$ is a vector in $\mathbb{R}^2$,
\[
    \lim_i \PR_{n,\beta^{(i)}} (x) = \left\{
    \begin{array}{ll}
        1 & \text{if }  x
     = v_{k,n}, \\
    0 & \text{otherwise.}\\
\end{array}
\right.
    \]
That is, $\PR_{n,\beta^{(i)}}$ converges in total variation to
    the point mass at $v_{k}$ as $i \rightarrow \infty$.
\end{corollary}
 \begin{proof}
    We only provide a brief sketch of the proof.
From Lemma \ref{lem:Pn}, $P_n$ is the convex hull of the points
$\{v_{k,n}, k=0,1\,\ldots,\lceil n/2
    \rceil -1\}$ and  $v_{n-1,n}$  and, for each fixed
    $k$, $v_{k,n} \rightarrow v_k$ as $n \rightarrow \infty$. Therefore, the
    normal cone to $v_{k,n}$ converges to $C_k$. Since by assumption $o \in
    C_k^{\circ}$, there
 exists an $n_0$, which depends on $o$ (and hence also on $k$), such
    that, for all $n>n_0$, $o$ is in the interior of the normal cone to
    $v_{k,n}$. The arguments used in the proof of Proposition
    \ref{prop:closure} yield the desired claim.
\end{proof}


%

\subsubsection*{Asymptotics of the closure of $\mathcal{E}_n$}
    We now study the asymptotic properties of the families
    $\mathcal{E}_{k,n}$ for fixed $k$ and as $n =
j(k+1)(k+2)$ for $j=1,2,\ldots$ tends to infinity.
    \begin{theorem}\label{thm:nj}
        Let $\{ n_j \}_{j=1,2,\ldots}$ be the sequence $n_j = j (k+1)(k+2)$.
        Then,
\[
    \lim_{j \rightarrow \infty} \frac{\PR_{n_j,k,\beta}(v_{k+1})}{\PR_{n_j,k,\beta}(v_k)}  \rightarrow
    \left\{
    \begin{array}{cl}
    \infty & \text{ if } \beta \in H^+_k \text{ or }  \beta \in H_k,\\
    0 & \text{ if } \beta \in H^-_k.\\
    \end{array}
    \right.
    \]
\end{theorem}

\begin{remark}
The proof further shows that the ratio of probabilities diverges
or vanishes at a rate exponential in $n_j^2$.
\end{remark}

\begin{proof}

 We can write
\[
    \frac{\PR_{n_j,k,\beta}(v_{k+1})}{\PR_{n_j,k,\beta}(v_k)}  = e^{ n^2 \langle
    l_k, \beta \rangle \langle v_{k+1} - v_k, l_k\rangle} \frac{ \nu_{n}(v_{k+1})
    }{ \nu_{n}(v_k)}.
    \]

We will first analyze the limiting behavior of the dominating
measure $\nu_{n}$. We will show that, as $n \rightarrow \infty$,
the number of Tur\'{a}n graphs with $r+1$  classes is larger than
the number of Tur\'{a}n graphs with $r$ classes by a
multiplicative factor that is exponential in $n$.

\begin{lemma}\label{lemma:vu.n}
    Consider the sequence of integers $n = j(k+1)(k+2)$, where $k \geq 1$ is a
    fixed integer and $j=1,2,\ldots$.
    Then, as $n \rightarrow \infty$,
\[
\frac{\nu_n(v_{k+1})}{\nu_n(v_{k})} \asymp  \sqrt{\frac{1 }{n}}
    \left( \frac{k+2}{k+1} \right)^n,
    \]
\end{lemma}
\begin{proof}
    Recall that $\nu_{n}(v_{k})$ is the number of (simple, labeled) graphs
    on $n$ nodes isomorphic to a Tur\'{a}n graph with $(k+1)$ classes each of size
    $j(k+2)$,
    and that $\nu_{n}(v_{k+1})$ is the number of (simple, labeled) graphs
    on $n$ nodes isomorphic to a Tur\'{a}n graph with $(k+2)$ classes each of size $j(k+1)$.
    Thus,
\[
    \nu_n(v_k) = \frac{1}{(k+1)!} \frac{n!}{ \left[ \left( j(k+2)  \right) !
    \right]^{k+1} }
    \]
    and
\[
    \nu_n(v_{k+1}) = \frac{1}{(k+2)!} \frac{n!}{ \left[ \left( j(k+1)  \right) !
    \right]^{k+2} }.
    \]
Next, since $n = j (k+1)(k+2)$, using Stirling's approximation we have that
\[
    \left( \left( j(k+2)  \right)!  \right)^{k+1} \sim (2 \pi j(k+2)
        )^{(k+1)/2}
    e^{- j(k+2)(k+1)} (j(k+2))^{j(k+2)(k+1)}\]
\[=(2 \pi j(k+2) )^{(k+1)/2}
    e^{- n} (j(k+2))^{n},
\]
and, similarly,
\[
    \left( \left( j(k+1)  \right)!  \right)^{k+2} \sim (2 \pi j(k+1)
        )^{(k+2)/2}
    e^{- j(k+1)(k+2)} (j(k+1))^{j(k+1)(k+2)}\]
\[=(2 \pi j(k+1) )^{(k+2)/2}
    e^{- n} (j(k+1))^{n}.
    \]
    Therefore,
    \begin{eqnarray*}
    \frac{\nu_n(v_{k+1})}{\nu_n(v_{k})} & \sim &  \frac{(k+1)!}{(k+2)!} \frac{
        (2 \pi)^{(k+1)/2}} {(2
    \pi)^{(k+2)/2}}\frac { \left( j(k+2) \right)^{(k+1)/2} }{\left( j(k+1)
        \right)^{(k+2)/2}
    } \frac{ \left( j(k+2) \right)^n}{ \left( j(k+1) \right)^n}\\
    & = &  \left[ \frac{(k+1)!}{(k+2)!} \frac{ (2 \pi)^{(k+1)/2}} {(2
        \pi)^{(k+2)/2}}
    \left( \frac{k+2}{k+1} \right)^{(k+1)/2} \sqrt{k+2}  \right] \sqrt{\frac{1 }{n}}
    \left( \frac{k+2}{k+1} \right)^n, \\
\end{eqnarray*}
    where we have used the fact that $j = \frac{n}{(k+1)(k+2)}$ for each $n$.
\end{proof}

Basic geometry considerations yield that, for any $\beta \in
\mathbb{R}^2$,
    \begin{equation*}\label{eq:sign}
\langle l_k, \beta \rangle \left\{
    \begin{array}{cr}
    >0 & \text{ if } \beta \in H^+_k,\\
    <0 & \text{ if } \beta \in H^-_k,\\
    =0 & \text{ if } \beta \in H_k.\\
    \end{array}
    \right.
\end{equation*}
Next, we have that
\[
    \langle v_{k+1} - v_{k}, l_k \rangle >0,
    \]
    since
    \[
    l_k  = \frac{1}{\sqrt{ 1 +\left( \frac{k(3k+5)}{(k+1)(k+2)} \right)^2 } } \left(
    \begin{array}{c}
            1\\
            \frac{k(3k+5)}{(k+1)(k+2)}
    \end{array} \right) \quad \text{and} \quad
    v_{k+1} - v_k = \left(
    \begin{array}{c}
    \frac{1}{(k+1)(k+2)}\\
    \frac{k(3k+5}{(k+1)^2(k+2)^2}
    \end{array}
    \right)
    \]
    are parallel vectors with positive entries.

 By Lemma \ref{lemma:vu.n}, we finally conclude that
\[
    \frac{\PR_{n_j,k,\beta}(v_{k+1})}{\PR_{n_j,k,\beta}(v_k)} \asymp  e^{n^2 C_k(\beta)} \sqrt{\frac{1 }{n}}
    \left( \frac{k+2}{k+1} \right)^n.
    \]
where $C_k(\beta) = \langle l_k, \beta \rangle \langle v_{k+1} -
v_k, l_k\rangle$. The result now follows since the term $e^{n^2
C_k(\beta)}$ dominates the other term and
$\mathrm{sign}(C_k(\beta)) = \mathrm{sign}( \langle l_k, \beta
\rangle)$.
\end{proof}

\begin{proof}[Proofs  of Theorems  \ref{thm:main.2},
    \ref{thm:main}, \ref{thm:main.3}, and \ref{thm:main.4}]
We first consider Theorem  \ref{thm:main}. Assume that $\beta \in
H^+_k$  or $\beta \in H_k$. Then by Theorem \ref{thm:nj}, there
exists an $n_0 = n_0(\beta, \epsilon ,k)$ such that, for all
$n>n_0$ and a multiple of $(k+1)(k+2)$,  $\PR_{n,k,\beta}(v_{k+1})
> 1 - \epsilon/2$ (recall that, by Proposition \ref{prop:minimal}, $ \PR_{n,k,\beta}(v_{k+1}) + \PR_{n,k,\beta}(v_{k})
= 1$). Let $n$ be an integer larger than
$n_0$ and a multiple of $(k+1)(k+2)$. By Proposition
\ref{prop:closure}, there exists an $r_0 =
r_0(\beta,\epsilon,k,n)$ such that, for all $r>r_0$, $\PR_{n,\beta
+ r o_k}(v_{k+1})
> \PR_{n,k,\beta}(v_{k+1}) - \epsilon/2$. Thus, for these values of
$n$ and $r$,
\[
    \PR_{n,\beta + r o_k}(v_{k+1}) > \PR_{n,k,\beta}(v_{k+1}) - \epsilon/2 > 1 -
    \epsilon/2 - \epsilon/2 = 1 - \epsilon,
    \]
as claimed. The case of $\beta \in H^-_k$ is proved in the same
way.

For Theorem \ref{thm:main.2}, we use Corollary \ref{cor:o}, which
guarantees that there exists an integer $n_0 = n_0(\beta,
\epsilon, o)$ such that, for any integer $n > n_0$, there exists
an $r_0 = r_0(\beta, \epsilon, o, n)$ such that for any $r>r_0$,
\[
    \PR_{n,\beta + r o}(v_{k,n}) > 1 - \epsilon.
    \]

Theorem \ref{thm:main.3} is proved as a direct corollary of
Proposition \ref{prop:minimal} along with simple algebra. Finally,
Theorem \ref{thm:main.4} follows from similar arguments used in
the proof of Proposition \ref{prop:minimal}.

\end{proof}

\begin{proof}[Proof of Theorem \ref{main.further}]
Subject to $\beta_1=a\beta_2+b$, the variational problem takes the following form: Find $f(x, y)$ so that
\begin{equation}
\beta_2 (ae+t^\gamma)+be-\iint_{[0,1]^2}I(f(x, y))dxdy
\end{equation}
is maximized, where $e=t(H_1, f)$ denotes the edge density and
$t=t(H_2, f)$ denotes the triangle density of $f$, respectively.
As in the proof of Theorem \ref{thm.at}, we see that as $\beta_2\rightarrow \infty$, the limiting optimizer $f^*$ must maximize
$ae+t^\gamma$. This implies that $f^*$ must lie on the curve $t=e^{3/2}$
(i.e., upper boundary of the feasible region) (see Figure
\ref{et}). Consider $g(e)=ae+e^{\frac{3}{2}\gamma}$. It is clear that if $\gamma\geq 2/3$, then $g''(e)\geq 0$ for $e\in (0,1)$, which implies that the maximizer $f^*$ is attained at either the empty graph or the complete graph. Further investigations show that same conclusions hold as in the standard model where $\gamma=1$. When $\gamma<2/3$, there are two situations. If $a\geq -\frac{3}{2}\gamma$, then $g'(e)\geq 0$ on $(0,1)$ always and the maximizer $f^*$ is given by the complete graph. If $a<-\frac{3}{2}\gamma$, then $g(e)$ is first increasing and then decreasing on $(0,1)$, and the optimal edge density $e^*$ satisfies $(e^*)^{\frac{3}{2}\gamma-1}=-\frac{2a}{3\gamma}$. This says that the maximizer $f^*$ has a non-trivial structure. It represents a complete subgraph coupled with isolated vertices, and the size of the complete subgraph is determined by $e^*$ \cite{L2}. We note that as $a$ decays from $-\frac{3}{2}\gamma$ to $-\infty$, the non-trivial graph transitions from being almost complete to almost empty.
\end{proof}


\begin{thebibliography}{99}

\bibitem{Aldous} Aldous, D.: Representations for partially exchangeable arrays of random
variables. J. Multivariate Anal. 11, 581-598 (1981)

\bibitem{BRN:78} Barndorff-Nielsen, O.: Information and Exponential Families in Statistical Theory.
John Wiley \& Sons, New York (1978)

\bibitem{B} Bhamidi, S., Bresler, G., Sly, A.: Mixing time of exponential random
graphs. Ann. Appl. Probab. 21, 2146-2170 (2011)

\bibitem{B:76} Bollob\'{a}s, B: Relations between sets of complete subgraphs. In: Nash-Williams, C., Sheehan, J.
(eds.) Proceedings of the Fifth British Combinatorial Conference,
Congressus Numerantium, No. XV, pp. 79-84. Utilitas Mathematica
Publishing, Winnipeg (1976)


\bibitem{Bo} Bollob\'{a}s, B.: Random Graphs. Volume 73 of Cambridge Studies in Advanced
Mathematics. 2nd ed. Cambridge University Press, Cambridge (2001)

\bibitem{BCLSV1} Borgs, C., Chayes, J., Lov\'{a}sz, L., S\'{o}s, V.T., Vesztergombi, K.:
Counting graph homomorphisms. In: Klazar, M., Kratochvil, J.,
Loebl, M., Thomas, R., Valtr, P. (eds.) Topics in Discrete
Mathematics, Volume 26, pp. 315-371. Springer, Berlin (2006)

\bibitem{BCLSV2} Borgs, C., Chayes, J.T., Lov\'{a}sz, L., S\'{o}s, V.T., Vesztergombi, K.:
Convergent sequences of dense graphs I. Subgraph frequencies,
metric properties and testing. Adv. Math. 219, 1801-1851 (2008)

\bibitem{BCLSV3} Borgs, C., Chayes, J.T., Lov\'{a}sz, L., S\'{o}s, V.T., Vesztergombi, K.:
Convergent sequences of dense graphs II. Multiway cuts and
statistical physics. Ann. of Math. 176, 151-219 (2012)

\bibitem{BROWN:86} Brown, L.: Fundamentals of Statistical Exponential Families.
IMS Lecture Notes -- Monograph Series, Volume 9. Institute of
Mathematical Statistics, Hayward (1986)

\bibitem{CD} Chatterjee, S., Diaconis, P.: Estimating and
understanding exponential random graph models. arXiv: 1102.2650 (2011)


\bibitem{CV} Chatterjee, S., Varadhan, S.R.S.: The large deviation principle for
the Erd\H{o}s-R\'{e}nyi random graph. European J. Combin. 32,
1000-1017 (2011)

\bibitem{CS:05} Csisz\'{a}r, I., Mat\'{u}\v{s}, F.:
Closure of exponential families. Ann. Probab. 33, 582-600 (2005)

\bibitem{CS:08} Csisz\'{a}r, I., Mat\'{u}\v{s}, F.:
Generalized maximum likelihood estimates for exponential families.
Probab. Theory Related Fields 141, 213-246 (2008)



\bibitem{DJ:08} Diaconis, P., Janson, S.: Graph limits and exchangeable
    random graphs. Rendiconti di Matematica 28, 33-61
    (2008)

\bibitem{Diestel} Diestel, R.: Graph Theory. Springer, New York (2005)


\bibitem{EN:11} Engstr\"{o}m, A., Nor\'{e}n, P.:
Polytopes from subraph statistics. arXiv: 1011.3552 (2011)

\bibitem{EKR:76} Erd\H{o}s, P., Kleitman, D.J., Rothschild, B.L.:
Asymptotic enumeration of $K_n$-free graphs. International
Colloquium on Combinatorial Theory, Volume 2, pp. 19-27. Atti dei
Convegni Lincei, Rome (1976)


\bibitem{F} Fadnavis, S.: Graph colorings and graph limits.
\\http://purl.stanford.edu/dr620kx9292 (2012)

\bibitem{F1} Fienberg, S.: Introduction to papers on the
modeling and analysis of network data. Ann. Appl. Statist. 4, 1-4
(2010)

\bibitem{F2} Fienberg, S.: Introduction to papers on the
modeling and analysis of network data II. Ann. Appl. Statist. 4,
533-534 (2010)

\bibitem{Fisher} Fisher, D.C.: Lower bounds on the number of triangles in a
    graph. J. Graph Theory 13, 505-512 (1989)

\bibitem{FS} Frank, O., Strauss, D.: Markov graphs. J. Amer. Statist. Assoc.
81, 832-842 (1986)

\bibitem{GEYER:09} Geyer, C.J.: Likelihood inference in exponential families and directions of recession.
Electron. J. Stat. 3, 259-289 (2009)

\bibitem{geyer-thompson} Geyer, C.J., Thompson, E.A.: Constrained Monte
Carlo maximum likelihood for dependent data. J. R. Statist. Soc. B
54, 657-699 (1992)

\bibitem{GXFA} Goldenberg, A., Zheng, A.X., Fienberg, S.E., Airoldi, E.M.:
A survey of statistical network models. Found. Trends Mach. Learn.
2, 129-233 (2009)

\bibitem{Goodman} Goodman, A.W.:
On sets of acquaintances and strangers at any party. Amer. Math.
Monthly 66, 778-783 (1959)


\bibitem{HJ} H\"{a}ggstr\"{o}m, O., Jonasson, J.: Phase transition in the random triangle
model. J. Appl. Probab. 36, 1101-1115 (1999)

\bibitem{H:03} Handcock, M.S.: Assessing degeneracy in statistical models of social networks.
Working paper no. 39, Center for Statistics and the Social
Sciences, University of Washington (2003)

\bibitem{HL} Holland, P., Leinhardt, S.: An exponential family of probability
distributions for directed graphs. J. Amer. Statist. Assoc. 76,
33-50 (1981)

\bibitem{Hoover} Hoover, D.: Row-column exchangeability and a generalized model for
probability. In: Koch, G., Spizzichino, F. (eds.) Exchangeability
in Probability and Statistics, pp. 281-291. North-Holland,
Amsterdam (1982)


\bibitem{ergm} Hunter, D.R., Handcock, M.S., Butts, C.T., Goodreau,
S.M., Morris, M.: ergm: A package to fit, simulate and diagnose
exponential-family models for networks. J. Statist. Softw. 24,
1-29 (2008)

\bibitem{Kallenberg:2005} Kallenberg, O.: Probabilistic Symmetries and
Invariance Principles. Springer, New York (2005)

\bibitem{Kol} Kolaczyk, E.D.: Statistical Analysis of Network Data: Methods and
Models. Springer, New York (2009)

\bibitem{Lauritzen:2003} Lauritzen, S.L.: Rasch models with exchangeable rows
and columns. In: Bernardo, J.M., Bayarri, M.J., Berger, J.O.,
Dawid, A.P., Heckerman, D., Smith, A.F.M., West, M. (eds.)
Bayesian Statistics 7, pp. 215-232. Oxford University Press,
Oxford (2003)

\bibitem{L1} Lov\'{a}sz, L.: Very large graphs.
In: Jerison, D., Mazur, B., Mrowka, T., Schmid, W., Stanley, R.,
Yau, S.T. (eds.) Current Developments in Mathematics, Volume 2008,
pp. 67-128. International Press, Boston (2009)

\bibitem{L2} Lov\'{a}sz, L.: Large Networks and Graph Limits.
American Mathematical Society, Providence (2012)

\bibitem{LS1} Lov\'{a}sz, L., Simonovits, M.: On the number of complete
    subgraphs of a graph II. In: Erd\H{o}s, P. et al. (eds.) Studies in Pure Mathematics, To the Memory of Paul
    Tur\'{a}n, pp. 459-495. Springer, Basel (1983)

\bibitem{LS} Lov\'{a}sz, L., Szegedy, B.: Limits of
dense graph sequences. J. Combin. Theory Ser. B 96, 933-957 (2006)



\bibitem{LZ} Lubetzky, E., Zhao, Y.: On replica symmetry of large deviations in random graphs. arXiv:
1210.7013 (2012)

\bibitem{Ma} Ma, S.-K.: Statistical Mechanics. World Scientific, Singapore (1985)


\bibitem{Newman.book} Newman, M.: Networks: An Introduction. Oxford University
    Press, New York (2010)

\bibitem{ERGM4} Pattison, P.E., Wasserman, S.: Logit models and logistic regressions for social networks: II. Multivariate relations.
Br. J. Math. Stat. Psychol. 52,
169-194 (1999)

\bibitem{Pik} Pikhurko, O.: An analytic approach to stability. Discrete
Math. 310, 2951-2964 (2010)

\bibitem{RS} Radin, C., Sadun, L.: Phase transitions in a complex
network. J. Phys. A 46, 305002 (2013)

\bibitem{RY} Radin, C., Yin, M.: Phase transitions in
exponential random graphs. Ann. Appl. Probab. 23, 2458-2471 (2013)

\bibitem{Razborov} Razborov, A.: On the minimal density of triangles in graphs.
Combin. Probab. Comput. 17, 603-618 (2008)

\bibitem{R} Rinaldo, A., Fienberg, S., Zhou, Y.: On the geometry of discrete
exponential families with application to exponential random graph
models. Electron. J. Stat. 3, 446-484 (2009)


\bibitem{ERGM1} Robins, G.L., Pattison, P.E., Kalish, Y., Lusher, D.: An
    introduction to exponential random graph (p*) models for
social networks. Social Networks 29, 173-191 (2007)

\bibitem{ERGM2} Robins, G.L., Pattison, P.E., Wasserman, S.: Logit models and logistic regressions for social networks: III. Valued
relations. Psychometrika 64, 371-394 (1999)

\bibitem{SR:13} Shalizi, C.R., Rinaldo, A.:
Consistency under sampling of exponential random graph models.
Ann. Statist. 41, 508-535 (2013)

\bibitem{S} Snijders, T., Pattison, P., Robins, G., Handcock, M.: New specifications for
exponential random graph models. Sociol. Method. 36, 99-153 (2006)



\bibitem{WF} Wasserman, S., Faust, K.: Social Network Analysis: Methods and
Applications. Cambridge University Press, Cambridge (2010)

\bibitem{ERGM3} Wasserman, S., Pattison, P.E.: Logit models and logistic
    regressions for social networks: I. An introduction to Markov graphs and
    p*. Psychometrika 61, 401-425 (1996)

\bibitem{WS} Watts, D., Strogatz, S.: Collective dynamics of
`small-world' networks. Nature 393, 440-442 (1998)

\bibitem{Yan.Xu} Yan, T., Xu, J.: A central limit theorem in the
    $\beta$-model for undirected random graphs with a diverging number of vertices.
    Biometrika 100, 519-524 (2013)

\end{thebibliography}
\end{document}